 \documentclass[final,1p,times]{elsarticle}

\usepackage{amssymb}
 \usepackage{amsthm}
\usepackage{amsmath,amssymb,amsopn,amsfonts,mathrsfs,amsbsy,amscd}
\usepackage{longtable}
\usepackage{longtable}
\usepackage{caption}
\usepackage{multirow}

\newcommand{\prs}{\langle\;,\;\rangle}

\newcommand{\too}{\longrightarrow}
\newcommand{\om}{\omega}

\newcommand{\esp}{\quad\mbox{and}\quad}

\def\br{[\;,\;]}

\newcommand{\G}{\mathfrak{g}}
\newcommand{\g}{\mathfrak{g}}

\newcommand{\h}{{\mathfrak{h}}}

\newcommand{\ad}{{\mathrm{ad}}}
\newcommand{\Ad}{{\mathrm{Ad}}}
\newcommand{\tr}{{\mathrm{tr}}}
\newcommand{\ric}{{\mathrm{ric}}}

\newcommand{\D}{{\cal D}}

\newcommand{\di}{\displaystyle}
\newcommand{\Om}{\Omega}
\newcommand{\na}{\nabla}

\newcommand{\al}{\alpha}

\newcommand{\ga}{\gamma}
\newcommand{\Ga}{\Gamma}
\newcommand{\e}{\epsilon}

\newcommand{\la}{\lambda}

\newtheorem{Def}{Definition}[section]
\newtheorem{theo}{Theorem}[section]
\newtheorem{pr}{Proposition}[section]

\newtheorem{co}{Corollary}[section]

\newtheorem{exem}{Example}
\newtheorem{remark}{Remark}

\font\bb=msbm10

\def\R{\hbox{\bb R}}

\begin{document}

\begin{frontmatter}
	
	
	
	
	\title{ On the Hermitian structures of the sequence of  tangent bundles of an affine manifold endowed with a Riemannian metric }
	
	\author[label1]{ Mohamed Boucetta}
	\address[label1]{Universit\'e Cadi-Ayyad\\
		Facult\'e des sciences et techniques\\
		BP 549 Marrakech Maroc\\e-mail: m.boucetta@uca.ac.ma  
		
	}
	
	
	
	
	
	\begin{abstract} Let $(M,\na,\prs)$ be a manifold endowed with a flat torsionless connection $\na$ and a Riemannian metric $\prs$ and $(T^kM)_{k\geq1}$ the sequence of tangent bundles given by $T^kM=T(T^{k-1}M)$ and $T^1M=TM$. We  show that, for any $k\geq1$,  $T^kM$ carries  a Hermitian structure $(J_k,g_k)$ and a flat torsionless connection $\na^k$ and when $M$ is a Lie group and $(\na,\prs)$ are left invariant there is a Lie group structure on each $T^kM$ such that $(J_k,g_k,\na^k)$ are left invariant.
		  It is  well-known  that $(TM,J_1,g_1)$ is K\"ahler if and only if $\prs$ is Hessian, i.e, in each system of affine coordinates $(x_1,\ldots,x_n)$, $\langle\partial_{x_i},\partial_{x_j}\rangle=\frac{\partial^2\phi}{\partial_{x_i}\partial_{x_j}}$.
	Having in mind many generalizations of the K\"ahler condition introduced recently,	we give the conditions on $(\na,\prs)$ so that $(TM,J_1,g_1)$ is   balanced,  locally conformally balanced, locally conformally K\"ahler, pluriclosed, Gauduchon, Vaismann or Calabi-Yau with torsion. Moreover, we can control at the level of $(\na,\prs)$ the conditions insuring that some $(T^kM,J_k,g_k)$ or all of them satisfy a generalized K\"ahler condition.
		For instance, we show  that there are some classes of $(M,\na,\prs)$ such that, for any $k\geq1$, $(T^kM,J_k,g_k)$ is  balanced non-K\"ahler and Calabi-Yau with torsion. By carefully studying the geometry of $(M,\na,\prs)$, we develop a powerful machinery to build a large classes of generalized K\"ahler manifolds.

	\end{abstract}
	
\end{frontmatter}
	
	\section{Introduction}\label{section1}

	Let $(N,J,g)$ be a complex manifold of real dimension $2n$, $n \geq 2$, equipped with a Hermitian metric $g$. For any $\eta\in\Om^p(N)$,
	\[ J\eta(X_1,\ldots,X_p)=(-1)^p\eta(JX_1,\ldots,JX_p)\esp d^c\eta=-(-1)^pJ^{-1}dJ\eta,\quad X_1,\ldots,X_p\in\Ga(TN). \]
	The {\it fundamental form} is given by $\om(.,.)=g(J.,.)$ and the {\it Lee form} is given by $\theta=Jd^*\om=-d^*\om\circ J$, where for any $X_1,\ldots,X_{p-1}\in\Ga(TN)$,
	\[ d^*\eta(X_1,\ldots,X_{p-1})=-\sum_{i=1}^{2n}\na^{LC}_{E_i}\eta(E_i,X_1,\ldots,X_{p-1}), \]	 $\na^{LC}$ is the Levi-Civita connection of $g$ and $(E_1,\ldots,E_{2n})$ is a local $g$-orthonormal frame.
	 A fundamental class of Hermitian metrics is provided by K\"ahler metrics, satisfying $d\om = 0$. In literature, many generalizations of the K\"ahler condition have been introduced. Indeed, $(N,J,g)$ is called:
	\begin{enumerate}\item  {\it strongly K\"ahler with torsion} or {\it pluriclosed} if $dd^c\om=0$, i.e., $dJdJ\om=0$,
		\item  {\it balanced} if $\theta=0$,
		\item  {\it locally conformally balanced} if $\theta$ is closed,
			\item  {\it Gauduchon} if $d^*\theta=0$,
		
		\item  {\it locally conformally K\"ahler} if  $d\om=\frac1{n-1}\theta\wedge\om$ and if, in addition,  $\na^{LC}\theta=0$ then it is called {\it Vaisman}.

	\end{enumerate}For general results about these generalized K\"ahler  metrics, we refer the reader to \cite{Alexandrov1, Angella, Arroyo, Bismut, Fino1, Gau1, Moroianu, Ugarte, Yang, Zhao1}.
	
	The Levi-Civita connection of $(N,g)$ is the only torsion free metric connection. In general, it does not preserve the complex structure $J$, this condition forcing the metric to be K\"ahler. Gauduchon proved in \cite{Gau} that there exists and affine line of canonical Hermitian connections (they preserve both $J$ and $g$) passing through the Bismut connection and the Chern connection.
	The {\it Bismut connection} $\na^B$ (also known as {\it Strominger connection})  is the unique Hermitian connection  with totally skew-symmetric torsion and the {\it Chern connection} $\na^C$ is the unique Hermitian connection whose torsion has trivial $(1,1)$-component. For any $X,Y,Z\in\Ga(TN)$,
	\begin{equation}\label{BC} \begin{cases}g(\na_X^BY,Z)=g(\na_X^{LC}Y,Z)+\frac12d\om(JX,JY,JZ),\\
	g(\na_X^CY,Z)=g(\na_X^{LC}Y,Z)-\frac12d\om(JX,Y,Z).
	\end{cases} \end{equation}Let $R^\tau(X,Y)=\na_{[X,Y]}^\tau-\na_X^\tau\na_Y^\tau+\na_Y^\tau\na_X^\tau$ be the curvature tensor of $\na^\tau$. The {\it Ricci form} of $\na^\tau$ is given, for any $X,Y\in\Ga(TN)$, by
	\begin{equation}\label{tau} \rho^\tau(X,Y)=\frac12\sum_{i=1}^{2n}g(R^\tau(X,Y)E_i,JE_i), \end{equation}where
	$(E_1,\ldots,E_{2n})$ is a local $g$-orthonormal frame. It is known \cite{Alexandrov1} that
	$\rho^C=\rho^B-dJ\theta$. Hermitian structures satisfying $\mathrm{Hol}^0(\na^B) \subset \mathrm{SU}(n)$, or equivalently $\rho^B = 0$, are known in literature as {\it Calabi-Yau with torsion} and appear in heterotic string theory, related to the Hull-Strominger system in six dimensions \cite{Hull, Li, Strominger}.

On the other hand,	let $(M,\na,\prs)$ be a manifold of dimension $n$ endowed with a flat torsionless connection $\na$ and a Riemannian metric $\prs$. Actually, $\na$ defines an affine structure on $M$, i.e.,  there exists on $M$ an atlas of charts such that all transition functions between charts are affine transformations of $\R^n$. Conversely, any affine atlas defines a flat torsionless connection.  We refer to the charts of this atlas as affine coordinates.

Through-out this paper, we call such triple $(M,\na,\prs)$ an {\it affine-Riemann manifold},  we denote by $D$ the Levi-Civita connection of $\prs$, $K(X,Y)=D_{[X,Y]}-D_XD_Y+D_YD_X$ its curvature,   $\ga$ the {\it difference tensor} and  $\ga^*$ its adjoint  given, for any $X,Y,Z\in\Ga(TM)$, by
 \begin{equation}\label{ga1} \ga_XY=D_XY-\na_XY\esp \langle\ga_X^*Y,Z\rangle=\langle Y,\ga_XZ\rangle. \end{equation} 
Their traces with respect to the metric are the vector fields given by
\begin{equation}\label{tr} \tr_{\prs}(\ga)=\sum_{i=1}^n\ga_{E_i}E_i\esp \tr_{\prs}(\ga^*)=\sum_{i=1}^n\ga_{E_i}^*E_i,  \end{equation}
where $(E_1,\ldots,E_n)$ is a local $\prs$-orthonormal frame.  The 1-form $\al$ given, for any $X\in\Ga(TM)$, by \begin{equation}\label{al}\al(X)=\langle \tr_{\prs}(\ga^*),X\rangle\end{equation} is closed (see Proposition \ref{pr1}) and it is known as  the  {\it first Koszul form} in the theory of Hessian manifolds. The vanishing of $\al$ is equivalent to the Riemannian volume being parallel with respect to $\na$. We introduce also the 1-form $\xi$ given, for any $X\in\Ga(TM)$, by \begin{equation}\label{xi}\xi(X)=\langle \tr_{\prs}(\ga),X\rangle.\end{equation}
We call $\xi$ the {\it adjoint Koszul form}. These two 1-forms  play an important role in this paper.

It is well-known (see \cite{Shima}) that there is a Hermitian structure $(J_1,g_1)$ on $TM$ canonically associated to $(M,\na,\prs)$ and $(TM,J_1,g_1)$ is K\"ahler if and only if $\prs$ is Hessian, i.e, in each system of affine coordinates $(x_1,\ldots,x_n)$ there exists a function $\phi$ such that $\langle\partial_{x_i},\partial_{x_j}\rangle=\frac{\partial^2\phi}{\partial_{x_i}\partial_{x_j}}$. This is equivalent to $\prs$ satisfying the Codazzi equation
	\begin{equation}\label{cod} \na_{X}(\prs)(Y,Z)=\na_{Y}(\prs)(X,Z),\quad X,Y,Z\in\Ga(TM). \end{equation}We will see that \eqref{cod} is equivalent to $\ga=\ga^*$.
Actually, there is also a flat torsionless connection $\na^1$ on $TM$ such that $\na^1J_1=0$. The affine-Riemann structure $(TM,\na^1,g_1)$ gives rise to a Hermitian structure $(TTM,J_2,g_2)$ and a flat torsionless connection $\na^2$ on $TTM$. By induction, we get a sequence of Hermitian structures $(T^kM,J_k,g_k)$ where $T^kM=T(T^{k-1}M)$ and $T^1M=TM$. Moreover, each $T^kM$ carries a flat torsionless connection $\na^k$ such that $\na^k(J_k)=0$.

The purpose of this paper is to explore the properties of  this sequence of Hermitian structures and find the conditions on $(M,\na,\prs)$ leading to some or all $(T^kM,J_k,g_k)$ to satisfy one of the generalized K\"ahler conditions introduced above. This will lead to the construction of  interesting classes of generalized K\"ahler  manifolds. We find also a large class of Hermitian manifolds which are Calabi-Yau with torsion or with vanishing Chern Ricci form. We will  show also that the study of the geometry of affine-Riemann manifolds is interesting in its own right and we will generalize some results obtained on Hessian manifolds.

Let us enumerate the main results of this paper and give its organization:
\begin{enumerate}\item In Section \ref{section2}, we define the sequence of Hermitian structures $(T^kM,J_k,g_k)$ and the sequence of flat torsionless connections $\na^k$ associated to an affine-Riemann manifold $(M,\na,\prs)$ and  we show that if $M$ is a Lie group and $(\na,\prs)$ are left invariant then there is on each $T^kM$ a Lie group structure such that $(J_k,g_k,\na^k)$ are left invariant. 
	\item In Section \ref{section3}, we give  the useful tools for the study of the Hermitian manifolds $(T^kM,J_k,g_k)$. Namely, we compute the Levi-Civita connection $\na^{LC}$ of $(TM,g_1)$ and  we show that the Lee form $\theta_1$ of $(TM,J_1,g_1)$ is given by means of the Koszul forms, namely, $\theta_1=\pi_1^*(\al-\xi)$ where $\pi_1:TM\too M$ is the canonical projection. We compute also the difference tensor for $(TM,\na^1,g_1)$ as well as it dual and we deduce by induction the Koszul forms $\al_k$ and $\xi_k$ and hence the Lee form of $(T^kM,\na^k,g_k)$. We give the conditions involving $\al,\xi,\tr_{\prs}(\ga),\tr_{\prs}(\ga^*)$ so that $(TM,J_1,g_1)$ is balanced, locally conformally balanced, Gauduchon, locally conformally K\"ahler or Vaisman.
	
	\item In Section \ref{sectionp}, we  prove that $(TM,J_1,g_1)$ is pluriclosed if and only if the curvature $K$ of $\prs$ satisfies, for any $X,Y\in\Ga(TM)$,
	\[ K(X,Y)=\ga_X^*\circ\ga_Y-\ga_Y^*\circ\ga_X. \]
	It is known (see \cite[Theorem 8.8 pp. 162]{Shima}) that if $M$ is compact, $\tr_{\prs}(\ga^*)=0$ and $\prs$ is Hessian, i.e., $(TM,J_1,g_1)$ is K\"ahler then $\na$ is the Levi-Civita of $\prs$.  By using the splitting theorem  of J. Cheeger and D. Gromoll (see for instance \cite[Corollary 6.67 pp. 168]{besse}), we prove that this result is still valid when we suppose that $M$ is compact, $\tr_{\prs}(\ga^*)=0$, $\na$ is complete and $(TM,J_1,g_1)$ is pluriclosed (see Theorem \ref{theplc}).
	\item In Section \ref{section3bis},
	 we compute  the Bismut connection $\na^{B}$  and the Chern connection $\na^{C}$ of $(TM,J_1,g_1)$ and we give their the curvatures  and their Ricci forms.  We show that if $\tr_{\prs}(\ga)=0$ (resp. $\tr_{\prs}(\ga^*)=0$) then the Ricci form $\rho^B$ (resp. $\rho^C$) of $(TM,J_1,g_1)$ vanishes.

	\item In Section \ref{section4},  we remark firs that if $\ga=0$ then for any $k\geq1$, the affine connection $\na^k$ is the Levi-Civita connection of $g_k$ and hence $(T^kM,J_k,g_k)$ is K\"ahler flat. Moreover, for $k_0\geq1$ fixed, we show:
	\begin{enumerate}\item If $k_0\geq2$ then  $(T^{k_0}M,J_{k_0},g_{k_0})$ is K\"ahler  if and only if $\ga=0$,
	\item   $(T^{k_0}M,J_{k_0},g_{k_0})$ is locally conformally balanced  if and only if $(TM,J_1,g_1)$ is locally conformally balanced and this is equivalent to $d\xi=0$,\item  $(T^{k_0}M,J_{k_0},g_{k_0})$ is balanced if and only if $\tr_{\prs}(\ga)=(2^{k_0}-1)\tr_{\prs}(\ga^*)$ and in this case all the others $(T^kM,J_k,g_k)$ are locally conformally balanced and $(T^{k_0+1}M,J_{k_0+1},g_{k_0+1})$ is Calabi-Yau with torsion. \end{enumerate}  We express also in affine local coordinates the conditions on $\prs$   so that $(TM,J_1,g_1)$ is  balanced or pluriclosed and we give many examples. We show that an affine-Riemann manifold $(M,\na,\prs)$ so that $(TM,J_1,g_1)$ is Vaisman non-K\"ahler carries a codimension one totally geodesic foliation and $\prs$ is flat when $\dim M=2$. 
	\item In Section \ref{section5}, we study the class of affine-Riemann manifolds $(M,\na,\prs)$ satisfying $D(\ga)=0$. We call the elements of this class {\it rigid affine-Riemann} manifolds.  We show for this class that, for any $k\geq1$, $(T^kM,J_k,g_k)$ is locally conformally balanced. By using a theorem by Kostant \cite[Theorem 4]{kostant}, we show that when $M$ is simply-connected and $\prs$ is complete then  $(M,\prs)$ is a symmetric space and there is a connected Lie group $G$ which acts transitively and reducibly on $M$ by preserving both $\na$ and $\prs$. We determine the elements of such class when $\dim M\leq 3$.
	\item In Section \ref{section6}, we study the class of affine-Riemann manifolds $(M,\na,\prs)$ satisfying $\tr_{\prs}(\ga)=\tr_{\prs}(\ga^*)=0$. We call the elements of this class {\it infinitely balanced affine-Riemann} manifolds. Indeed, we prove that  the condition $\tr_{\prs}(\ga)=\tr_{\prs}(\ga^*)=0$ and $\ga\not=0$ imply
 that, for any $k\geq1$, $(T^kM,J_k,g_k)$ is  balanced (non-K\"ahler when $k\geq2$) and their Ricci forms $\rho_k^C$ and $\rho^B_k$ vanishes. Moreover, if $\ga=\ga^*$, i.e, $\prs$ is Hessian then the Ricci curvature of $\prs$ is nonnegative and $(TM,J_1,g_1)$ is K\"ahler Ricci-flat. In dimension 2, we show that $\tr_{\prs}(\ga)=\tr_{\prs}(\ga^*)=0$ implies that $\prs$ is Hessian and the curvature is nonnegative. Moreover, if $\prs$ is complete then $\ga=0$, i.e., $\na$ is the Levi-Civita of $\prs$. 
 Non trivial examples of   infinitely balanced affine-Riemann manifolds   exist. We 
	 show that, for any $n\geq2$, there is an affine-Riemann structure $(\na,\prs)$ on $S^n\times\R$ such that the corresponding $\ga$ satisfies, $\ga\not=0$, $\ga=\ga^*$ and $\tr_{\prs}(\ga)=0$. Thus, $(T(S^n\times\R),J_1,g_1)$ is K\"ahler Ricci-flat and, for any $k\geq2$, $(T^k(S^n\times\R),J_k,g_k)$ is  balanced non-K\"ahler  and their Ricci forms $\rho_k^C$ and $\rho^B_k$ vanishes.
	
	\item In Section \ref{section7}, by using the classification of 3-dimensional real Novikov algebras, we give the infinitesimal part of a large class of generalized K\"ahler left invariant structures on some 6-dimensional Lie groups. We give also a large class of Calabi-Yau with torsion left invariant  structures.
	
	\item We think that one of the important contribution of this work is the development of a powerful machinery which permits the construction of large classes of examples of generalized K\"ahler manifolds (see Theorems \ref{dim2}, \ref{dimn}, \ref{exemple}, Corollaries \ref{cop1}, \ref{co2}, Examples \ref{exem1}-\ref{exem5} and Tables \ref{table3}-\ref{table8}).

	\end{enumerate}

	\section{The canonical sequence of Hermitian structures associated to an affine-Riemann manifold}\label{section2}
	
	In this section,  we introduce  the Hermitian structures and the affine connections on the sequence of tangent bundles associated to an affine-Riemann manifold and we show that these structures are left invariant when the affine-Riemann structure is left invariant.

	\subsection{The Hermitian structures on the sequence of tangent bundles associated to an affine-Riemann manifold}
	Let $(M,\na,\prs)$ be an { affine-Riemann} manifold of dimension $n$. Let $\pi_1:TM\too M$ be the canonical projection and $Q:TTM\too TM$ the connection map of $\na$ locally given   by
	\[ Q\left( \sum_{i=1}^n b_i\partial_{x_i}+\sum_{j=1}^nZ_j\partial_{\mu_j}\right)=
	\sum_{l=1}^n\left( Z_l+\sum_{i=1}^n\sum_{j=1}^n b_i\mu_j\Ga_{ij}^l\right)\partial_{x_l}, \]where $(x_1,\ldots,x_n)$ is a system of local coordinates,   $(x_1,\ldots,x_n,\mu_1,\ldots,\mu_n)$ the associated system of coordinates on $TM$ and $\na_{\partial_{x_i}} \partial_{x_j}=\sum_{l=1}^n\Ga_{ij}^l \partial_{x_l}$. Then
	\[ TTM=\ker T\pi_1\oplus \ker Q. \]
	For  $X\in\Ga(TM)$, we denote by $X^h$ its horizontal lift and by $X^v$ its vertical lift.  The flow of $X^v$ is given by $\Phi^X(t,(x,u))=(x,u+tX(x))$ and  $X^h(x,u)=h^{(x,u)}(X(x))$, where $h^{(x,u)}:T_xM\too \ker Q(x,u)$ is the inverse of the restriction of $d\pi_1$ to $\ker Q(x,u)$. Since the curvature of $\na$ vanishes, for any $X,Y\in\Ga(TM)$,
	\begin{equation}\label{br} [X^h,Y^h]=[X,Y]^h,\;[X^h,Y^v]=(\na_XY)^v \esp[X^v,Y^v]=0. \end{equation} 
	 The connection ${\na^1}$  on  $TM$ given  by
		\begin{equation}\label{con} {\na^1}_{X^h}Y^h=(\na_XY)^h,\;{\na^1}_{X^h}Y^v=(\na_XY)^v\esp
		{\na^1}_{X^v}Y^h=\na^1_{X^v}Y^v=0, \end{equation}for any $X,Y\in\Ga(TM)$, is flat torsionless and defines an affine structure on $TM$. The tensor field   $J_1:TTM\too TTM$ given by $J_1X^h=X^v$ and $J_1X^v=-X^h$ satisfies $J_1^2=-\mathrm{Id}_{TTM}$, $\na^1(J_1)=0$  and hence defines a complex structure on $TM$.

	On the other hand, we define on $TM$ a Riemannian metric $g_1$ by
	\[ g_1(X^h,Y^h)=\langle X,Y\rangle\circ\pi_1,\; g_1(X^v,Y^v)=\langle X,Y\rangle\circ\pi_1\esp g_1(X^h,Y^v)=0,\quad X,Y\in\Ga(TM).\] 
	This metric is Hermitian with respect to $J_1$ and its fundamental form $\om=g_1(J_1.,.)$ satisfies
	\begin{equation}\label{om} \om(X^h,Y^h)=\om(X^v,Y^v)=0\esp \om(X^h,Y^v)=-\om(Y^v,X^h)=\langle X,Y\rangle\circ\pi_1,\quad X,Y\in\Ga(TM). \end{equation}
	 Actually, we have a sequence of Hermitian structures. The affine-Riemann manifold $(TM,\na^1,g_1)$ gives rise to a Hermitian structure $(TTM,J_2,g_2)$ and a flat torsionless connection $\na^2$ on $TTM$. By induction, we get a sequence of Hermitian structures $(T^kM,J_k,g_k)$ where $T^kM=T(T^{k-1}M)$ and $T^1M=TM$. Moreover, each $T^kM$ carries a flat torsionless connection $\na^k$ such that $\na^k(J_k)=0$.

	\subsection{The canonical sequence of Hermitian structures  associated to a left invariant affine-Riemann structure}
	
	Let $(G,\na,\prs)$ be an affine-Riemann manifold such that $G$ is a connected Lie group and $(\na,\prs)$ are left invariant. Let $(\G=T_eG,\br)$ be the Lie algebra of $G$. For any $a\in\G$, we denote by $a^-$ the left invariant vector field on $G$ associated to $a$. The affine connection $\na$ defines  a product $\bullet$ on $\G$  by
	\[ (a\bullet b)^-=(\na_{a^-}b^-)(e),\quad a,b\in\G. \]This product is Lie-admissible, i.e., $a\bullet b-b\bullet a=[a,b]$ and left symmetric, i.e., for any $a,b,c\in\G$,
	\[ \mathrm{ass}(a,b,c)=\mathrm{ass}(b,a,c), \] where $\mathrm{ass}(a,b,c)=(a\bullet b)\bullet c-a\bullet(b\bullet c)$.
	This is equivalent to  $\mathrm{L}:\G\too\mathrm{End}(\G)$, $a\mapsto \mathrm{L}_a$ is a representation of Lie algebras, where $\mathrm{L}_ab=a\bullet b$, 
	
	Put
	$\Phi(\G)=\G\times\G$ and define on $\Phi(\G)$ a product $\star$,   a  bracket $\br_\Phi$, an isomorphism $J:\Phi(\G)\too\Phi(\G)$ and a scalar product $\prs_\phi$  by
	\begin{eqnarray}\label{tg} &&(a,b)\star(c,d)=(a\bullet c,a\bullet d),\quad  [(a,b),(c,d)]_\Phi=([a,c],a\bullet d-c\bullet b), \\
	\label{J}  &&J(a,b)=(-b,a)\esp \langle(a,b),(c,d)\rangle_\Phi=\langle a,c\rangle+\langle b,d\rangle,\end{eqnarray}
	for any $(a,b),(c,d)\in\Phi(\G)$.
	It is easy to check that $\star$ is left symmetric and hence its commutator which is $\br_\Phi$  is a Lie bracket. 
	Moreover, for any $(a,b),(c,d)\in\Phi(\G)$,
	\[ N_{J}((a,b),(c,d))=[J(a,b),J(c,d)]_\Phi-J[(a,b),J(c,d)]_\Phi-J[J(a,b),(c,d)]_\Phi-
	[(a,b),(c,d)]_\Phi=0. \]
	On the other hand, let $\rho:G\too\mathrm{GL}(\G)$ be the homomorphism of groups such that $d_e\rho=\mathrm{L}$ and consider the product on $G\times\G$ given by 
		\[ (p,a).(q,b)=(pq,b+\rho(q^{-1})(a)),\quad p,q\in G,a,b\in\G. \]
		
		\begin{pr} $(G\times\G,.)$ is a Lie group  whose Lie algebra is $(\Phi(\G),\br_\Phi)$.\end{pr}
		\begin{proof} For any $(p,a),(q,b)\in G\times\G$,
		  $(p,a)^{-1}=(p^{-1},-\rho(p)(a))$ and
	\begin{align*}L_{(p,a)}\circ R_{(p^{-1},-\rho(p)(a))}(q,b)&=(p,a)(q,b)
	(p^{-1},-\rho(p)(a))\\
	&=(pq,b+\rho(q^{-1})(a))(p^{-1},-\rho(p)(a))\\
	&=(pqp^{-1},-\rho(p)(a)+\rho(p)(b)+\rho(pq^{-1})(a)).
	\end{align*}So, for any $(X,b),(Y,c)\in T_{(e,0)}(G\times\G)$,
	\[ \Ad_{(p,a)}(X,b)=(\Ad_pX,\rho(p)(b)-\rho(p)(X\bullet a)) \]and hence
	\[ [(X,b),(Y,c)]=([X,Y],X\bullet c-Y\bullet b). \]
	\end{proof}

	The triple $(J,\prs_\Phi,\star)$ induces a left invariant triple $(J_0,g_0,\na^0)$  on $G\times\G$ satisfying
	\[ J_0(a,b)^-=(-b,a)^-,\;  
	g_0((a,b)^-,(c,d)^-)=\langle(a,b),(c,d)\rangle_\Phi\esp \na^0_{(a,b)^-}(c,d)^-=\left( (a,b)\star(c,d) \right)^-,\quad a,b,c,d\in\G. \]
	Thus $(G\times\G,J_0,g_0)$ is a left invariant Hermitian structure and $\na^0$ is a left invariant flat torsionless connection on $G\times\G$.
	
	Denote by $\Theta:TG\too G\times\G$ the identification $X_p\too (p,T_pL_{p^{-1}}X_p)$.

	\begin{theo}\label{main} Let $(TG,J_1,g_1)$ be the canonical Hermitian structure associated to $(G,\na,\prs)$ and $\na^1$ the associated canonical affine connection. Then $\Theta$ sends $(J_1,g_1,\na^1)$ to $(J_0,g_0,\na^0)$, i.e., for any $X,Y\in \Ga(TG)$,
		\[ g(X,Y)=g_0(\Theta_{*}X,\Theta_{*}Y), \Theta_{*} (J_1X)=J_0\Theta_{*}X\esp \Theta_{*}(\na^1_XY)=\na^0_{\Theta_{*}X}\Theta_{*}Y. \]
		
	\end{theo}

	To prove this theorem, we need some preparation.
	\begin{pr}\label{par1} Let $(G,\D)$ be a Lie group endowed with a left invariant connection,
		 $\tau:[0,1]\too G$ a curve and  $V:[0,1]\too TG$  a vector field along $\tau$.  We define $\tau^\ell:[0,1]\too \G$ and $W:[0,1]\too\G$ by
		\[ \tau^\ell(t)=T_{\tau(t)}L_{\tau(t)^{-1}}(\tau'(t))\esp W(t)=T_{\tau(t)}L_{\tau(t)^{-1}}(V(t)). \]
		Then $V$ is parallel along $\tau$ with respect $\D$ if and only if, for any $t\in[0,1]$,
		\[ W'(t)+\tau^\ell(t)\bullet W(t)=0, \]
		where $u\bullet v=(\D_{u^-}v^-)(e)$.
	\end{pr}
	
	\begin{proof} We consider $(u_1,\ldots,u_n)$ a basis of $\G$ and $(X_1,\ldots,X_n)$ the corresponding left invariant vector fields. Then
		\[ 
		\tau^\ell(t)=\sum_{i=1}^n\tau^\ell_i(t)u_i,\;W(t)=\sum_{i=1}^nW_i(t)u_i,\;
		 \tau'(t)=\sum_{i=1}^n\tau^\ell_i(t)X_i,\;V(t)=\sum_{i=1}^nW_i(t)X_i.
		 \]	Then
		\begin{eqnarray*}
			\D_{t}V(t)&=&\sum_{i=1}^nW_i'(t)X_i+
			\sum_{i=1}^nW_i(t)\D_{\tau'(t)}X_i\\
			&=&\sum_{i=1}^nW_i'(t)X_i+\sum_{i,j=1}^nW_i(t)\tau^\ell_j(t)\D_{X_j}X_i\\
			&=&\sum_{i=1}^nW_i'(t)X_i+\sum_{i,j=1}^nW_i(t)\tau^\ell_j(t)(u_j\bullet u_i)^-\\
			&=&\left(W'(t)+\tau^\ell(t)\bullet W(t)\right)^-
		\end{eqnarray*}and the result follows.
	\end{proof}

	\begin{pr}\label{hl} Let $(G,\na,\prs)$ be a left invariant affine-Riemann structure on a connected Lie group. Then:\begin{enumerate}
		\item For any $X\in T_pG$ and any $a\in\G$,
		\[ T\Theta(X^v)(p,a)=(0,T_pL_{p^{-1}}(X))\esp T\Theta(X^h)(p,a)=(X,-T_pL_{p^{-1}}(X)\bullet a). \]
		\item For any $(a,b)\in\G\times\G$,  
		$ (a,b)^-=T\Theta((a^-)^h+(b^-)^v)$.
		\end{enumerate}
		
	\end{pr}
	\begin{proof}\begin{enumerate}\item The first relation is obvious. Recall that the horizontal lift of $X$ at $u_p\in TG$ is given by
		\[ X^h(u_p)=\frac{d}{dt}_{|t=0}V(t), \]where $V:[0,1]\too TG$ is the $\na$-parallel vector field along a curve
		$\tau:[0,1]\too G$  such that $\tau(0)=p$,  $\tau'(0)=X$ and $V(0)=u_p$. Put $a=T_pL_{p^{-1}}(u_p)$. By virtue of Proposition \ref{par1},
		\[ T_{u_p}\Theta(X^h)=\frac{d}{dt}_{|t=0}(\tau(t),W(t))=(X,-T_pL_{p^{-1}}(X)\bullet a). \]
		
	\item For any $p\in G$ and $u\in\G$, we have
		\begin{eqnarray*} (a,b)^-(p,u)&=&T_{(e,0)}L_{(p,u)}(a,b)\\&=&
			\frac{d}{dt}_{|t=0}(p,u)(\exp(ta),tb)\\&=& 
			\frac{d}{dt}_{|t=0}(p\exp(ta),tb+\rho(\exp(-ta))(u))\\&=&
			(a^-(p),b-a\bullet u)\\
			&=&(a^-(p),-T_pL_{p^{-1}}(a^-(p))\bullet u)+(0,T_pL_{p^{-1}}(b^-(p))\\
			&=&T\Theta(a^-)^h(p,u)+T\Theta(b^-)^v(p,u).
		\end{eqnarray*}\end{enumerate}
		
	\end{proof}
	
	\paragraph{Proof of Theorem \ref{main} }\begin{proof} By virtue of Proposition \ref{hl},
		\begin{align*}
		J_0(a,b)^-&=(-b,a)^-=-T\Theta(b^-)^h+T\Theta(a^-)^v,\\
		&=T\Theta J_1((b^-)^v)+T\Theta J_1((a^-)^h)\\
		&=T\Theta J_1(T\Theta)^{-1} (a,b)^-.
		\end{align*}The other relations can be deduced similarly.
	\end{proof}

	\section{Basic tools for the  study of the canonical sequence of Hermitian structures associated to an affine-Riemann manifold}\label{section3}

	Trough-out this section and the next one,  $(M,\na,\prs)$ is an affine-Riemann manifold of dimension $n$,  $D$ the Levi-Civita connection of $\prs$ and $\mu$ the Riemannian volume. Let $(T^kM,J_k,g_k,\na^k)$, $k\geq1$, be the canonical sequence of Hermitian structures associated to $(M,\na,\prs)$ endowed with   the sequence of flat torsionless connections. For any $k\geq1$, we denote by  $\pi_k:T^{k}M\too T^{k-1}M$ the canonical projection.
	We consider the difference tensor  $\ga$ and its dual $\ga^*$  given by \eqref{ga1}, their traces given by \eqref{tr} and the Koszul forms $\al$ and $\xi$ given by \eqref{al} and \eqref{xi}. 
	
	Since both $\na$ and $D$ are torsionless, $\ga$ is symmetric and it is easy to check that, for any $X,Y,Z,U\in\Ga(TM)$,
	 \begin{equation}\label{nam}
	 \na_{X}(\prs)(Y,Z)=\langle\ga_XY+\ga^*_XY,Z\rangle,\,\langle D_X(\ga)(Y,Z),U\rangle=\langle D_X(\ga^*)(Y,U),Z\rangle.
	 \end{equation}Since  $\na$ is flat, the curvature $K(X,Y)=D_{[X,Y]}-[D_X,D_Y]$ of $\prs$ satisfies
	 \begin{equation}\label{K} K(X,Y)Z=D_Y(\ga)(X,Z)-D_X(\ga)(Y,Z)+[\ga_X,\ga_Y]Z. \end{equation}
	 From the relation $K(X,Y)^*=-K(X,Y)$, we deduce that
	\begin{equation}\label{cu1} K(X,Y)Z=D_X(\ga^*)(Y,Z)-D_Y(\ga^*)(X,Z)
	+[\ga_X^*,\ga_Y^*]Z\end{equation}and hence
	\begin{equation}\label{cu}
	D_Y(\ga+\ga^*)(X,Z)-D_X(\ga+\ga^*)(Y,Z)=[\ga_X^*,\ga_Y^*]Z-[\ga_X,\ga_Y]Z.
	\end{equation}
	 The first Koszul form $\al$ satisfies the following properties.
	 \begin{pr}\label{pr1} For any $X\in\Ga(TM)$,
	 	\begin{equation}\label{nu} \na_X\mu=\langle X,\tr_{\prs}(\ga*)\rangle\mu=\tr(\ga_X)\mu=\al(X)\mu. \end{equation}
	 	In particular, the first Koszul 1-form $\al$   is closed.
	 \end{pr}
	 \begin{proof}Let $(E_1,\ldots,E_n)$ be a  local $\prs$-orthonormal frame.
	 	\begin{align*}
	 	\na_X\mu(E_1,\ldots,E_n)
	 	&=-\sum_{i=1}^n\mu(E_1,\ldots,\na_XE_i,\ldots,E_n)
	 	=-\sum_{i=1}^n\langle\na_XE_i,E_i\rangle
	 	=\frac12\sum_{i=1}^n\na_{X}(\prs)(E_i,E_i)\\
	 	&\stackrel{\eqref{nam}}=\frac12\langle\ga_X{E_i}+\ga_X^*E_i,E_i\rangle
	 	=\tr(\ga_X)
	 	=\langle\tr_{\prs}(\ga^*),X\rangle.
	 	\end{align*}The fact that $\al$ is closed is a consequence of  the fact that, for any $X,Y\in\Ga(TM)$, $$(\na_{[X,Y]}-\na_X\na_Y+\na_Y\na_X)\mu=0.$$ 
	 		\end{proof}	
	 \begin{pr}\label{prdom}The differential of the  fundamental form $\om$ associated to $(TM,J_1,g_1)$ is given by
	 	\[ d\om(X^h,Y^h,Z^h)=d\om(X^v,Y^v,Z^v)=d\om(X^h,Y^v,Z^v)=0\esp
	 	d\om(X^h,Y^h,Z^v)=
	 	\langle \ga_X^*Y-\ga_Y^*X,Z\rangle\circ\pi_1, \]for any $X,Y,Z\in\Ga(TM)$. Hence
	 	\[ (J_1d\om)(X^h,Y^h,Z^h)=(J_1d\om)(X^v,Y^v,Z^v)=
	 	(J_1d\om)(X^h,Y^h,Z^v)=0\esp 
	 	(J_1d\om)(X^v,Y^v,Z^h)=-\langle \ga_X^*Y-\ga_Y^*X,Z\rangle\circ\pi_1.
	 	 \]
	 	
	 \end{pr}
	 \begin{proof}
	 	From \eqref{br} and \eqref{om}, we have obviously $d\om(X^h,Y^h,Z^h)=d\om(X^v,Y^v,Z^v)=d\om(X^h,Y^v,Z^v)=0$. On the other hand,
	 	\begin{align*}	
	 	d\om(X^h,Y^h,Z^v)&= X.\langle Y,Z\rangle\circ\pi_1-Y.\langle X,Z\rangle\circ\pi_1-\langle[X,Y],Z\rangle\circ\pi_1-\langle \na_XZ,Y\rangle\circ\pi_1+\langle \na_YZ,X\rangle\circ\pi_1\\
	 	&=\na_X(\prs)(Y,Z)\circ\pi_1-\na_Y(\prs)(X,Z)\circ\pi_1,\\
	 	&\stackrel{\eqref{nam}}=\langle \ga_X^*Y-\ga_Y^*X,Z\rangle\circ\pi_1.
	 	\end{align*}\end{proof}	
	 As an immediate consequence of the expression of $d\om$, the proof above,	\eqref{cu1} and \eqref{cu}, we get the following result which  sum up some of the important properties of Hessian manifolds (see \cite{Shima}). Recall that  a Hessian manifold is an affine-Riemann manifold $(M,\na,\prs)$ such that in any affine coordinates $(x_1,\ldots,x_n)$ there exists a function $\phi$ such that $\langle\partial_{x_i},\partial_{x_j}\rangle=\frac{\partial^2\phi}{\partial x_i\partial x_j}$ for any $i,j\in\{1,\ldots,\}$. This is equivalent to $\prs$ satisfying the Codazzi equation \eqref{cod}.
	\begin{co}\label{coh}$(TM,J_1,g_1)$ is K\"ahler if and only if $(M,\na,\prs)$ is Hessian manifold which is also equivalent to $\ga=\ga^*$.  In this case, 
		\[D_Y(\ga)(X,Z)=D_X(\ga)(Y,Z)\esp  K(X,Y)=[\ga_X,\ga_Y],\quad X,Y,Z\in\Ga(TM). \]

		\end{co} 
		
	Let us compute now the Levi-Civita connection $\na^{LC}$ of $(TM,J_1,g_1)$ and its Lee form.
	\begin{pr}\label{prLC}For any $X,Y\in\Ga(TM)$,
		\[ \na^{LC}_{X^h}Y^h=(D_XY)^h,\; \na^{LC}_{X^v}Y^v=-\frac12(\ga_X^*Y+\ga_Y^*X)^h,\;\na^{LC}_{X^v}Y^h=
		(\ga_Y^sX)^v \esp \na^{LC}_{X^h}Y^v=(D_XY)^v-(\ga_X^aY)^v,\] where
		\[ \ga^a=\frac12(\ga-\ga^*)\esp\ga^s=\frac12(\ga+\ga^*). \]

	\end{pr}
	\begin{proof} Let $X,Y,Z\in\Ga(TM)$. From the Koszul formula and \eqref{br},
		we have obviously
		\[2g(\na^{LC}_{X^h}Y^h,Z^v)=2g(\na^{LC}_{X^v}Y^v,Z^v)=
		2g(\na^{LC}_{X^v}Y^h,Z^h)=0.  \]On the other hand,
		\begin{align*}2g(\na^{LC}_{X^h}Y^h,Z^h)&=2\langle D_XY,Z\rangle\circ\pi_1,\\
			2g(\na^{LC}_{X^v}Y^v,Z^h)&=-Z.\langle X,Y\rangle\circ\pi_1+\langle\na_ZX,Y\rangle\circ\pi_1+\langle\na_ZY,X\rangle\circ\pi_1\\
		&=-\na_Z(\prs)(X,Y)\circ\pi_1\\
		&\stackrel{\eqref{nam}}=-\langle \ga_ZX+\ga_Z^*X,Y\rangle\circ\pi_1,\\
		&=-\langle \ga_X^*Y+\ga_Y^*X,Z\rangle\circ\pi_1,\\
		2g(\na^{LC}_{X^v}Y^h,Z^v)&=Y.\langle X,Z\rangle\circ\pi_1-\langle\na_YX,Z\rangle\circ\pi_1-\langle\na_YZ,X\rangle\circ\pi_1,\\
		&=\na_Y(\prs)(X,Z)\circ\pi_1\\&\stackrel{\eqref{nam}}=\langle \ga_YX+\ga_Y^*X,Z\rangle\circ\pi_1.
		\end{align*}  
	\end{proof}

	\begin{pr}\label{prlee} The Lee form $\theta_1$ of $(TM,J_1,g_1)$ is given by
		\begin{equation}\label{lee} \theta_1=\pi_1^*(\al-\xi), \end{equation}where $\al$ and $\xi$ are the Koszul forms of $(M,\na,\prs)$.
		In particular, $(TM,J_1,g_1)$ is balanced if and only if $\al=\xi$ which is also equivalent to
		\begin{equation} \label{bl}\tr_{\prs}(\ga^*-\ga)=0. \end{equation}
		Moreover, $(TM,J_1,g_1)$ is locally conformally balanced if and only if $d\xi=0$.

		\end{pr}
	\begin{proof}Let $(E_1,\ldots,E_n)$ be a local $\prs$-orthonormal frame. Having in mind the expressions of $\na^{LC}$ given in the last proposition, for any $X\in\Ga(TM),$
		\begin{align*}
		-d^*\om(X^h)&=\sum_{i=1}^n\left( \na^{LC}_{E_i^h}\om(E_i^h,X^h)+\na^{LC}_{E_i^v}\om(E_i^v,X^h)\right) \\
		&=\sum_{i=1}^n\left(E_i^h.\om(E_i^h,X^h)-\om(\na^{LC}_{E_i^h}E_i^h,X^h)-\om(E_i^h,\na^{LC}_{E_i^h}X^h)+
		E_i^v.\om(E_i^v,X^h)-\om(\na^{LC}_{E_i^v}E_i^v,X^h)-\om(E_i^v,\na^{LC}_{E_i^v}X^h)\right)\\
		&=0.\end{align*}
			\begin{align*}
		-d^*\om(X^v)&=\sum_{i=1}^n\left( \na^{LC}_{E_i^h}\om(E_i^h,X^v)+\na^{LC}_{E_i^v}\om(E_i^v,X^v) \right) \\
		&=\sum_{i=1}^n\left( E_i^h.\om(E_i^h,X^v)-\om(\na^{LC}_{E_i^h}E_i^h,X^v)-\om(E_i^h,\na^{LC}_{E_i^h}X^v)+
		E_i^v.\om(E_i^v,X^v)-\om(\na^{LC}_{E_i^v}E_i^v,X^v)
		-\om(E_i^v,\na^{LC}_{E_i^v}X^v)\right)\\
		&=\sum_{i=1}^n\left(E_i.\langle E_i,X\rangle\circ\pi_1-\langle D_{E_i}E_i,X\rangle\circ\pi_1-\langle E_i,D_{E_i}X\rangle\circ\pi_1+\frac12\langle E_i,\ga_{E_i}X-\ga_{E_i}^*X\rangle\circ\pi_1\right.\\&\left.
		+\langle\ga_{E_i}^*E_i,X\rangle\circ\pi_1-\frac12\langle\ga_{E_i}^*X+\ga_X^*E_i,E_i\rangle\circ\pi_1\right)\\
		&=\langle\tr_{\prs}(\ga^*)-\tr_{\prs}(\ga),X\rangle\circ\pi_1.
		\end{align*}Finally,
		\[ \theta_1(X^h)=-d^*\om(J_1X^h)=\langle\tr_{\prs}(\ga^*)-\tr_{\prs}(\ga),X\rangle\circ\pi_1,\;\;\theta_1(X^v)=0 \]and we get the desired formula. Moreover, since $\al$ is closed then $d\theta_1=0$ if and only if $d\xi=0$.
		\end{proof}
		
		\begin{pr} \label{gaud} We have
			\[ d^*\theta_1=d^*(\al-\xi)\circ\pi_1-\langle\tr_{\prs}(\ga^*)-\tr_{\prs}(\ga),
			\tr_{\prs}(\ga^*)\rangle\circ\pi_1 \]and hence
			$(TM,J_1,g_1)$ is Gauduchon if and only if
			\begin{equation}\label{g}
			d^*(\al-\xi)=|\tr_{\prs}(\ga^*)|^2-\langle \tr_{\prs}(\ga^*) ,\tr_{\prs}(\ga) \rangle.
			\end{equation}
			
		\end{pr}
		\begin{proof} A straightforward computation using the definition of the divergence and the expressions of $\na^{LC}$. \end{proof}
		
		\begin{pr}\label{vaisman}\begin{enumerate}\item $(TM,J_1,g_1)$ is locally conformally K\"ahler if and only if, for any $X,Y\in\Ga(TM)$,
			\begin{equation}\label{LCK} (n-1)(\ga_X^*Y-\ga_Y^*X)= \theta_0(X)Y-\theta_0(Y) X, \end{equation}where $\theta_0=\al-\xi$.
			
			\item $(TM,J_1,g_1)$ is Vaisman if and only if \eqref{LCK} holds and the vector field $\Pi:=\tr_{\prs}(\ga^*-\ga)$ is parallel with respect to both $D$ and $\na$.
			
			\end{enumerate}
			
		\end{pr}
		
		\begin{proof}\begin{enumerate}\item $(TM,J_1,g_1)$ is locally conformally K\"ahler if and only if $(n-1)d\om=\theta_1\wedge\om$. For any $X,Y,Z\in\Ga(TM)$,
			\begin{align*}
			\theta_1\wedge\om(X^h,Y^h,Z^h)&=\theta_1\wedge\om(X^v,Y^v,Z^v)=0,\\
			\theta_1\wedge\om(X^h,Y^h,Z^v)&=\theta_1(X^h)\langle Y,Z\rangle\circ\pi_1-\theta_1(Y^h)\langle X,Z\rangle\circ\pi_1\\
			&=\theta_0(X)\circ\pi_1\langle Y,Z\rangle\circ\pi_1-\theta_0(Y)\circ\pi_1\langle X,Z\rangle\circ\pi_1,\\
			\theta_1\wedge\om(X^h,Y^v,Z^v)&=0
			\end{align*}and the first assertion follows by virtue of  Proposition \ref{prdom}.
			\item $(TM,J_1,g_1)$ is Vaisman if and only if it is locally conformally K\"ahler and $\na^{LC}\theta_1=0$. Now, by using Proposition \ref{prLC}, for any $X,Y\in\Ga(TM)$,
			\[ \na^{LC}_{X^h}(\theta_1)(Y^h)=(D_X(\theta_0)(Y))^h,\;\na^{LC}_{X^h}(\theta_1)(Y^v)=\na^{LC}_{X^v}(\theta_1)(Y^h)=0\esp \na^{LC}_{X^v}(\theta_1)(Y^v)=
			\frac12\theta_0(\ga_X^*Y+\ga_Y^*X).
			 \]On the other hand, if  \eqref{LCK} holds then $\theta_0(\ga_X^*Y-\ga_Y^*X)=0$. Thus $(TM,J_1,g_1)$ is Vaisman if and only if \eqref{LCK} holds, $D\theta_0=0$ and, for any $X,Y\in\Ga(TM)$, $\theta_0(\ga_X^*Y)=0$. This relation is equivalent to $\ga_X(\Pi)=0$ for any $X\in\Ga(TM)$,  $D\theta_0=0$ is equivalent to $D\Pi=0$ and we get the desired result since $\ga=D-\na$. \qedhere
			\end{enumerate}
				\end{proof}
		
		Let us compute the difference tensor $\Ga=\na^{LC}-{\na^1}$ of $(TM,\na^1,g_1)$ as well as its adjoint $\Ga^*$,  the Koszul forms $\al_k$,  $\xi_k$ as well as the Lee form $\theta_k$ of $(T^kM,\na^k,g_k)$.
		\begin{pr}\label{Gaa} For any $X,Y\in\Ga(TM)$,
			\begin{align*}\Ga_{X^h}Y^h&=(\ga_XY)^h,
			\Ga_{X^h}Y^v=\Ga_{Y^v}X^h=
			\frac12(\ga_XY+\ga_X^*Y)^v,
			\Ga_{X^v}Y^v=-\frac12(\ga_X^*Y+\ga_Y^*X)^h,\\
			\Ga_{X^h}^*Y^h&=(\ga_X^*Y)^h, \Ga_{X^h}^*Y^v=\frac12(\ga_XY+\ga_X^*Y)^v,\;\Ga_{X^v}^*Y^h=-\frac12(\ga_XY+\ga_Y^*X)^v\esp \Ga_{X^v}^*Y^v=\frac12(\ga_X^*Y+\ga_Y^*X)^h,\\
			\tr_{g_1}(\Ga)&=(\tr_{\prs}(\ga)-\tr_{\prs}(\ga^*))^h,\; 
			\tr_{g_1}(\Ga^*)=2(\tr_{\prs}(\ga^*))^h,\;\\  
			\xi_k&=-\theta_k,\; \al_k=2^k\pi_k^*\circ\ldots\circ\pi_1^*(\al)\esp \theta_k=\pi_k^*\circ\ldots\circ\pi_1^*((2^k-1)\al-\xi),\quad k\geq1.\nonumber
			\end{align*}						
			
		\end{pr}
		
		\begin{proof} The expressions of $\Ga$ are an immediate consequence of Proposition \ref{prLC} and \eqref{con} and one can deduce  easily $\Ga^*$. If $(E_1,\ldots,E_n)$ is a local $\prs$-orthonormal frame then
			\[ \tr_{g_1}(\Ga)=\sum_{i=1}^n\left(\Ga_{E_i^h}E_i^h+\Ga_{E_i^v}E_i^v  \right)=(\tr_{\prs}(\ga)-\tr_{\prs}(\ga^*))^h\esp 
			\tr_{g_1}(\Ga^*)=\sum_{i=1}^n\left(\Ga_{E_i^h}^*E_i^h+\Ga_{E_i^v}^*E_i^v  \right)=2(\tr_{\prs}(\ga^*))^h. \]
			This implies that $\xi_1=\pi_1^*(\xi-\al)=-\theta_1$,  $\al_1=2\pi_1^*(\al)$ and hence 
			$$\theta_2=\pi_2^*(\al_1-\xi_1)=\pi_2^*\circ\pi_1(3\al-\xi),\;
			\xi_2=\pi_2^*(\xi_1-\al_1)=-\theta_2\esp \al_2=2\pi_2^*(\al_1)=2^2\pi_2^*\circ\pi_1(\al).$$
			By induction, we get all the desired formulas.
			\end{proof}
	We end this section by a remark on the  Hermitian structure $(TM,J_1,g_1)$. Indeed,	the fact that $\na^1(J_1)=0$ makes the Hermitian structure $(TM,J_1,g_1)$ particular as the following remark suggests. We don't use this remark in our paper but, may be, it can be used in further studies.
		\begin{remark}One can check that the tensor $\Ga$ satisfies, for any $U,V\in\Ga(TTM)$, \begin{equation} \label{eqG}{\Ga_{J_1U}J_1V-J_1\Ga_{J_1U}V-J_1\Ga_UJ_1V-\Ga_UV=0.} \end{equation}
			By using the fact that $\na^1(J_1)=0$ and the known formula (see \cite[Proposition 4.2]{K})
			\[ 4g_1(\na^{LC}_U(J_1)V,W)=6d\om(U,J_1V,J_1W)-6d\om(U,V,W)+
			g_1(N_{J_1}(V,W),J_1U) \]we get that
			\eqref{eqG} is equivalent to
			\begin{equation*}d\om(J_1U,J_1V,J_1W)=d\om(J_1U,V,W)+d\om(U,J_1V,W)+d\om(U,V,J_1W).
			\end{equation*}

			\end{remark}
		
		\section{Affine-Riemann manifolds with pluriclosed $(TM,J_1,g_1)$}\label{sectionp}
		
	In this section, we give the conditions so that $(TM,J_1,g_1)$ is pluriclosed and we generalize a result obtained in the theory of Hessian manifolds. 
	
		Let us compute $dd^c\om=-dJ_1^{-1}dJ_1\om=dJ_1d\om$.
		
		\begin{pr}\label{prdc} For any $X,Y,Z,U\in\Ga(TM)$,
			\[ \begin{cases}dJ_1d\om(X^h,Y^h,Z^h,U^h)=dJ_1d\om(X^v,Y^v,Z^v,U^v)=
			dJ_1d\om(X^h,Y^h,Z^h,U^v)=dJ_1d\om(X^h,Y^v,Z^v,U^v)=0,\\
			dJ_1d\om(X^h,Y^h,Z^v,U^v)=2\langle K(X,Y)Z-(\ga_X^*\circ \ga_Y-\ga_Y^*\circ\ga_X)Z,U\rangle\circ\pi_1.
			\end{cases} \]In particular, $(TM,J_1,g_1)$ is pluriclosed if and only if for any $X,Y\in\Ga(TM)$, the curvature $K$ of $D$ satisfies
			\begin{equation}\label{pl} K(X,Y)=\ga_X^*\circ \ga_Y-\ga_Y^*\circ\ga_X. \end{equation}
			
		\end{pr}
		
		\begin{proof}Put $\nu(X,Y)=\ga_X^*Y-\ga_Y^*X$.
			By using Proposition \ref{prdom}, we get easily
			\[ dJ_1d\om(X^h,Y^h,Z^h,U^h)=dJ_1d\om(X^v,Y^v,Z^v,U^v)=
			dJ_1d\om(X^h,Y^h,Z^h,U^v)=dJ_1d\om(X^h,Y^v,Z^v,U^v)=0. \]On the other hand, having in mind \eqref{br}, let us compute $S:=dJ_1d\om(X^h,Y^h,Z^v,U^v)$. Indeed,
			\begin{align*}
			S&=X^h.J_1d\om(Y^h,Z^v,U^v)-Y^h.J_1d\om(X^h,Z^v,U^v)-
			J_1d\om([X,Y]^h,Z^v,U^v)+J_1d\om((\na_XZ)^v,Y^h,U^v)\\&-J_1d\om((\na_XU)^v,Y^h,Z^v)
			- J_1d\om((\na_YZ)^v,X^h,U^v)+J_1d\om((\na_YU)^v,X^h,Z^v)\\
			&=-X.\langle\nu(Z,U),Y\rangle\circ\pi_1+Y.\langle\nu(Z,U),X\rangle\circ\pi_1
			+\langle\nu(Z,U),[X,Y]\rangle\circ\pi_1+\langle\nu(\na_XZ,U),Y\rangle\circ\pi_1
			\\&+\langle\nu(Z,\na_XU),Y\rangle\circ\pi_1
			-\langle\nu(\na_YZ,U),X\rangle\circ\pi_1-\langle\nu(Z,\na_YU),X\rangle\circ\pi_1\\
			&=-X.\langle\ga_Z^*U,Y\rangle\circ\pi_1+X.\langle\ga_U^*Z,Y\rangle\circ\pi_1
			+Y.\langle\ga_Z^*U,X\rangle\circ\pi_1-Y.\langle\ga_U^*Z,X\rangle\circ\pi_1
			+\langle\ga_Z^*U,[X,Y]\rangle\circ\pi_1\\
			&-\langle\ga_U^*Z,[X,Y]\rangle\circ\pi_1
			+\langle\ga_{D_XZ}^*U,Y\rangle\circ\pi_1-\langle\ga_{U}^*D_XZ,Y\rangle\circ\pi_1
			-\langle\ga_{\ga_XZ}^*U,Y\rangle\circ\pi_1+\langle\ga_{U}^*\ga_XZ,Y\rangle\circ\pi_1\\
			&+\langle\ga_Z^*D_XU,Y\rangle\circ\pi_1-\langle\ga_{D_XU}^*Z,Y\rangle\circ\pi_1
			-\langle\ga_Z^*\ga_XU,Y\rangle\circ\pi_1+\langle\ga_{\ga_XU}^*Z,Y\rangle\circ\pi_1\\
			&-\langle\ga_{D_YZ}^*U,X\rangle\circ\pi_1+\langle\ga_{U}^*D_YZ,X\rangle\circ\pi_1
			+\langle\ga_{\ga_YZ}^*U,X\rangle\circ\pi_1-\langle\ga_{U}^*\ga_YZ,X\rangle\circ\pi_1\\
			&-\langle\ga_Z^*D_YU,X\rangle\circ\pi_1+\langle\ga_{D_YU}^*Z,X\rangle\circ\pi_1
			+\langle\ga_Z^*\ga_YU,X\rangle\circ\pi_1-\langle\ga_{\ga_YU}^*Z,X\rangle\circ\pi_1.
			\end{align*} We simplify this expression by using the properties of $D$:
			\begin{align*}
			S&=-\langle D_X\ga_Z^*U,Y\rangle\circ\pi_1+\langle D_X\ga_U^*Z,Y\rangle\circ\pi_1
			+\langle D_Y\ga_Z^*U,X\rangle\circ\pi_1-\langle D_Y\ga_U^*Z,X\rangle\circ\pi_1\\
			&+\langle\ga_{D_XZ}^*U,Y\rangle\circ\pi_1-\langle\ga_{U}^*D_XZ,Y\rangle\circ\pi_1
			-\langle\ga_{\ga_XZ}^*U,Y\rangle\circ\pi_1+
			\langle\ga_{U}^*\ga_XZ,Y\rangle\circ\pi_1\\
			&+\langle\ga_Z^*D_XU,Y\rangle\circ\pi_1-\langle\ga_{D_XU}^*Z,Y\rangle\circ\pi_1
			-\langle\ga_Z^*\ga_XU,Y\rangle\circ\pi_1+\langle\ga_{\ga_XU}^*Z,Y\rangle\circ\pi_1\\
			&-\langle\ga_{D_YZ}^*U,X\rangle\circ\pi_1+\langle\ga_{U}^*D_YZ,X\rangle\circ\pi_1
			+\langle\ga_{\ga_YZ}^*U,X\rangle\circ\pi_1-\langle\ga_{U}^*\ga_YZ,X\rangle\circ\pi_1\\
			&-\langle\ga_Z^*D_YU,X\rangle\circ\pi_1+\langle\ga_{D_YU}^*Z,X\rangle\circ\pi_1
			+\langle\ga_Z^*\ga_YU,X\rangle\circ\pi_1-\langle\ga_{\ga_YU}^*Z,X\rangle\circ\pi_1\\
			&=-\langle D_X(\ga^*)(Z,U),Y\rangle\circ\pi_1+\langle D_X(\ga^*)(U,Z),Y\rangle\circ\pi_1
			+\langle D_Y(\ga^*)(Z,U),X\rangle\circ\pi_1-\langle D_Y(\ga^*)(U,Z),X\rangle\circ\pi_1\\
			&
			-\langle\ga_{\ga_XZ}^*U,Y\rangle\circ\pi_1+\langle\ga_{U}^*\ga_XZ,Y\rangle\circ\pi_1
			-\langle\ga_Z^*\ga_XU,Y\rangle\circ\pi_1+\langle\ga_{\ga_XU}^*Z,Y\rangle\circ\pi_1\\
			&
			+\langle\ga_{\ga_YZ}^*U,X\rangle\circ\pi_1-\langle\ga_{U}^*\ga_YZ,X\rangle\circ\pi_1
			+\langle\ga_Z^*\ga_YU,X\rangle\circ\pi_1-\langle\ga_{\ga_YU}^*Z,X\rangle\circ\pi_1.\end{align*}
			By using \eqref{nam}, we get
			\begin{align*}
			S&=-\langle D_X(\ga)(Z,Y),U\rangle\circ\pi_1+\langle D_X(\ga)(U,Y),Z\rangle\circ\pi_1
			+\langle D_Y(\ga)(Z,X),U\rangle\circ\pi_1-\langle D_Y(\ga)(U,X),Z\rangle\circ\pi_1\\
			&
			-\langle U,\ga_Y\circ\ga_XZ\rangle\circ\pi_1+\langle\ga_{Y}^*\circ\ga_XZ,U\rangle\circ\pi_1
			-\langle U,\ga_X^*\circ\ga_YZ\rangle\circ\pi_1+\langle Z,\ga_Y\circ\ga_XU\rangle\circ\pi_1\\
			&
			+\langle U,\ga_X\circ\ga_YZ\rangle\circ\pi_1-\langle\ga_{X}^*\circ\ga_YZ,U\rangle\circ\pi_1
			+\langle U,\ga_Y^*\circ\ga_XZ\rangle\circ\pi_1-\langle\ga_Y^*\circ\ga_{X}^*Z,U\rangle\circ\pi_1\\
			&\stackrel{\eqref{K}}=\langle K(X,Y)Z,U\rangle\circ\pi_1-\langle K(X,Y)U,Z\rangle\circ\pi_1+2\langle(\ga_Y^*\circ \ga_X-\ga_X^*\circ\ga_Y)Z,U\rangle\circ\pi_1\\
			&=2\langle K(X,Y)Z-(\ga_X^*\circ \ga_Y-\ga_Y^*\circ\ga_X),U\rangle\circ\pi_1.
			\end{align*}
			\end{proof} 
			
			In \cite[Theorem 8.8 pp. 162]{Shima}, Shima proved that if $(M,\na,\prs)$ is a compact Hessian manifold such that its first Koszul form vanishes then $\na$ is the Levi-Civita connection of $\prs$. Note that in this case the first Koszul form and the dual Koszul form coincide. The following theorem is a generalization of this result under an additional assumption, namely, $\na$ is complete. It will be interesting to see if we can drop this assumption.
			
			\begin{theo}\label{theplc} Let $(M,\na,\prs)$ be an affine-Riemann manifold such that $(TM,J_1,g_1)$ is pluriclosed and the dual Koszul form of $(M,\na,\prs)$ vanishes. Then the Ricci curvature of $\prs$ is nonnegative. Moreover, if $M$ is compact and $\na$ is complete then $\ga=0$, i.e., $\na$ is the Levi-Civita connection of $\prs$.
				
				\end{theo}
				
			\begin{proof} Note that the vanishing of dual Koszul form is equivalent to $\tr_{\prs}(\ga)=0$. From the relation
				\[ K(X,Y)=\ga_X^*\circ\ga_Y-\ga_Y^*\circ\ga_X \]and the fact that
				$\tr_{\prs}(\ga)=0$, we deduce that the Ricci curvature of $\prs$ is given by
				\[ \ric(X,X)=\tr(\ga_X^*\circ\ga_X)\geq0 \]and $\ric(X,X)=0$ if and only if $\ga_X=0$.
				By using the splitting theorem  of J. Cheeger and D. Gromoll (see for instance \cite[Corollary 6.67 pp. 168]{besse}), we deduce that if $M$ is compact its universal Riemannian covering is isometric to a Riemannian product $(\overline{M}\times\R^d,\prs_1\times\prs_0)$ where is $\overline{M}$ is compact and $\prs_0$ is the canonical metric of $\R^d$. But if $\na$ is complete the universal covering of $M$ is diffeomorphic to $\R^n$ which completes the proof.
				\end{proof}	
				
				\section{The Bismut and Chern connections of $(TM,J_1,g_1)$ and their curvatures}\label{section3bis}
			
			Let $(M,\na,\prs)$ be  an affine-Riemann manifold of dimension $n$. The expressions of the Bismut and Chern connections of $(TM,J_1,g_1)$ can be deduced easily from \eqref{BC} and Propositions \ref{prdom}-\ref{prLC}.
			\begin{pr}\label{CB}We have, for any $X,Y\in\Ga(TM)$,
				\[\begin{cases} \na^B_{X^h}Y^h=(D_XY)^h,\;\na^B_{X^v}Y^v=-(\ga_Y^*X)^h,\\\;\na^B_{X^v}Y^h=
				(\ga_Y^*X)^v,\; \na^B_{X^h}Y^v=(D_XY)^v. \end{cases}\esp \begin{cases}\na^C_{X^h}Y^h=(D_XY)^h-(\ga^a_XY)^h,\;
				\na^C_{X^v}Y^v=-(\ga^s_XY)^h,\;\\\na^C_{X^v}Y^h=(\ga^s_XY)^v,\;
				\na^C_{X^h}Y^v=(D_XY)^v
				-(\ga^a_XY)^v.\end{cases} \]where
				\[ \ga^a=\frac12(\ga-\ga^*)\esp\ga^s=\frac12(\ga+\ga^*). \]	
			\end{pr}
			
			Now, we give the curvature $R^B$ of  $\na^B$.
			\begin{pr} \label{prCB} For any $X,Y,Z\in\Ga(TM)$,
				\[ \begin{cases}R^B(X^h,Y^h)Z^h=(K(X,Y)Z)^h,\;
				R^B(X^h,Y^h)Z^v=(K(X,Y)Z)^v,\\
				R^B(X^v,Y^v)Z^v=(\ga_{\ga_Z^*Y}^*X)^v-(\ga_{\ga_Z^*X}^*Y)^v,\;
				R^B(X^v,Y^v)Z^h=(\ga_{\ga_Z^*Y}^*X)^h-(\ga_{\ga_Z^*X}^*Y)^h,\\
				R^B(X^h,Y^v)Z^v=
				(\ga_{Z}^*\circ\ga_XY)^h+(D_X(\ga^*)(Z,Y))^h
				,\\
				R^B(X^h,Y^v)Z^h=
				-(\ga_Z^*\circ\ga_{X}Y))^v-(D_X(\ga^*)(Z,Y))^v
				,
				\end{cases} \]where $K$ is the curvature of $D$. Moreover, the Ricci form is given by
				\[ \rho^B(X^h,Y^h)=\rho^B(X^v,Y^v)=0\esp \rho^B(X^h,Y^v)=
					-\langle\ga_XY,\tr_{\prs}\ga\rangle\circ\pi_1-
					\langle\tr_{\prs}(D_X(\ga)),Y\rangle\circ\pi_1.     \]
				
			\end{pr}
			
			\begin{proof} We have
				\begin{align*}R^B(X^h,Y^h)Z^h&=(K(X,Y)Z)^h\\
				R^B(X^h,Y^h)Z^v&=(K(X,Y)Z)^v,\\
				R^B(X^v,Y^v)Z^v&=\na^B_{X^v}(\ga_Z^*Y)^h-\na^B_{Y^v}(\ga_Z^*X)^h=(\ga_{\ga_Z^*Y}^*X)^v-(\ga_{\ga_Z^*X}^*Y)^v,\\
				R^B(X^v,Y^v)Z^h&=-\na^B_{X^v}(\ga_Z^*Y)^v+\na^B_{Y^v}(\ga_Z^*X)^v=(\ga_{\ga_Z^*Y}^*X)^h-(\ga_{\ga_Z^*X}^*Y)^h,\\
				R^B(X^h,Y^v)Z^v&=\na^B_{(\na_{X}Y)^v}Z^v+\na^B_{X^h}(\ga_{Z}^*Y)^h+\na^B_{Y^v}(D_XZ)^v\\
				&=-(\ga_Z^*(\na_{X}Y))^h+(D_X\ga_{Z}^*Y)^h-(\ga_{D_XZ}^*Y)^h,\\
				&=(\ga_{Z}^*\circ\ga_XY)^h+(D_X\ga_{Z}^*Y)^h-(\ga_{D_XZ}^*Y)^h-(\ga_Z^*(D_{X}Y))^h,\\
				R^B(X^h,Y^v)Z^h&=\na^B_{(\na_{X}Y)^v}Z^h-\na^B_{X^h}(\ga_Z^*Y)^v+\na^B_{Y^v}(D_XZ)^h\\
				&=(\ga_Z^*(\na_{X}Y))^v-(D_X\ga_Z^*Y)^v+(\ga_{D_XZ}^*Y)^v\\
				&=-(\ga_Z^*(\ga_{X}Y))^v+(\ga_Z^*(D_{X}Y))^v-(D_X\ga_Z^*Y)^v+(\ga_{D_XZ}^*Y)^v.
				\end{align*}
				Let $(E_1,\ldots,E_n)$ be a local orthonormal frame of $\prs$. Then
				\begin{align*}
				2\rho^B(X^h,Y^h)&=\sum_{i=1}^n\left(g_1(R^B(X^h,Y^h)E_i^h,E_i^v)-g_1(R^B(X^h,Y^h)E_i^v,E_i^h)\right)=0,\\
				2\rho^B(X^v,Y^v)&=\sum_{i=1}^n\left(g_1(R^B(X^v,Y^v)E_i^h,E_i^v)-g_1(R^B(X^v,Y^v)E_i^v,E_i^h)\right)=0,\\
				2\rho^B(X^h,Y^v)&=\sum_{i=1}^n\left(g_1(R^B(X^h,Y^v)E_i^h,E_i^v)-g_1(R^B(X^h,Y^v)E_i^v,E_i^h)\right)\\
				&=\sum_{i=1}^n\left(-2\langle \ga_{E_i}^*\circ\ga_{X}Y+D_X(\ga^*)(E_i,Y),E_i\rangle\right)\\
				&=-2\langle\ga_{X}Y,\tr_{\prs}\ga\rangle\circ\pi_1-2\sum_{i=1}^n	\langle D_X(\ga^*)(E_i,Y),E_i\rangle\circ\pi_1\\
				&\stackrel{\eqref{nam}}=-2\langle\ga_{X}Y,\tr\ga\rangle\circ\pi_1-2\sum_{i=1}^n	\langle D_X(\ga)(E_i,E_i),Y\rangle\circ\pi_1\\
				&=-2\langle\ga_XY,\tr_{\prs}\ga\rangle\circ\pi_1-2
				\langle\tr_{\prs}(D_X(\ga)),Y\rangle\circ\pi_1.		 \end{align*}
			\end{proof}	
			
			\begin{pr}\label{pryau} Let $(M,\na,\prs)$ be an affine-Riemann manifold. If $\tr_{\prs}(\ga)=0$ then, for any $X\in\Ga(TM)$, $\tr_{\prs}(D_X(\ga))=0$. In particular, if $\tr_{\prs}(\ga)=0$ then $(TM,J_1,g_1)$ is Calabi-Yau with torsion, i.e., $\rho^B=0$.
				
				\end{pr}
				
				\begin{proof} Fix a point $p\in M$. It is known that there exists   a local orthonormal frame $(E_1,\ldots,E_n)$ in a neighborhood of $p$ such that $(DE_j)(p)=0$ for $j=1,\ldots,n$. Now suppose that $\tr_{\prs}(\ga)=0$. For any $X\in\Ga(TM)$,
					\begin{align*}\tr_{\prs}(D_X(\ga))&=\sum_{i=1}^nD_X(\ga)(E_i,E_i)\\
					&=D_X(\tr_{\prs}(\ga))-2\ga_{D_XE_i}E_i\\
					&=-2\ga_{D_XE_i}E_i.
					\end{align*} By evaluating at $p$ we get the result. The second assertion is a consequence of  this result and the expression of $\rho^B$ given in Proposition \ref{prCB}.
					\end{proof}
			
				We give also the curvature $R^C$ of  $\na^C$.
			\begin{pr}\label{prCC} For any $X,Y,Z\in\Ga(TM)$,
				\[ \begin{cases}R^C(X^h,Y^h)Z^h=R^C(X^v,Y^v)Z^h=([\ga_X^s,\ga_Y^s]Z)^h,\\
				R^C(X^h,Y^h)Z^v=R^C(X^v,Y^v)Z^v=([\ga_X^s,\ga_Y^s]Z)^v,\\
				R^C(X^h,Y^v)Z^h=-(D_X(\ga^s)(Y,Z))^v+([\ga^a_X,\ga^s_Y]Z)^v-(\ga^s_{\ga_XY}Z)^v,\\
				R^C(X^h,Y^v)Z^v=(D_X(\ga^s)(Y,Z))^h-([\ga^a_X,\ga^s_Y]Z)^h+(\ga^s_{\ga_XY}Z)^h,
				\end{cases} \]where $K$ is the curvature of $D$. Moreover, the Ricci form is given by
				\[ \rho^C(X^h,Y^h)=\rho^C(X^v,Y^v)=0\esp \rho^C(X^h,Y^v)=-\langle\ga_{X}Y,\tr_{\prs}(\ga^*)\rangle\circ\pi_1
				-\langle\tr_{\prs}(D_X(\ga^*)),Y\rangle\circ\pi_1.  \]
				
			\end{pr}\begin{proof}
			By using Proposition \ref{CB},
			\begin{align*}
			R^C(X^h,Y^h)Z^h&=(D_{[X,Y]}Z)^h-(\ga_{[X,Y]}^aZ)^h-\na^C_{X^h}(D_YZ)^h
			+\na^C_{X^h}(\ga^a_YZ)^h+\na^C_{Y^h}(D_XZ)^h
			-\na^C_{Y^h}(\ga^a_XZ)^h\\
			&=(D_{[X,Y]}Z)^h-(\ga_{[X,Y]}^aZ)^h-(D_XD_YZ)^h+(\ga^a_XD_YZ)^h+(D_X\ga^a_YZ)^h-(\ga^a_X\ga^a_YZ)^h\\&
			+(D_YD_XZ)^h-(\ga^a_YD_XZ)^h-(D_Y\ga^a_XZ)^h+(\ga^a_Y\ga^a_XZ)^h\\
			&=(K(X,Y)Z)^h+(D_X(\ga^a)(Y,Z))^h-(D_Y(\ga^a)(X,Z))^h+([\ga^a_Y,\ga^a_X]Z)^h.\end{align*}\begin{align*}
			R^C(X^v,Y^v)Z^v&=\na_{X^v}^C(\ga^s_YZ)^h-\na_{Y^v}^C(\ga^s_XZ)^h\\
			&=(\ga^s_X\ga^s_YZ)^v-(\ga^s_Y\ga^s_XZ)^v=([\ga^s_X,\ga^s_Y]Z)^v,\\
			R^C(X^v,Y^v)Z^h&=-\na_{X^v}^C(\ga^s_YZ)^v+\na_{Y^v}^C(\ga^s_XZ)^v\\
			&=([\ga^s_X,\ga^s_Y]Z)^h.\end{align*}
			\begin{align*}
			R^C(X^h,Y^v)Z^h&=\na^C_{(\na_XY)^v}Z^h-\na^C_{X^h}(\ga^s_YZ)^v+\na^C_{Y^v}(D_XZ)^h
			-\na^C_{Y^v}(\ga^a_XZ)^h\\
			&=(\ga^s_{D_XY}Z)^v-(\ga^s_{\ga_XY}Z)^v
			-(D_X\ga^s_YZ)^v+(\ga^a_X\ga^s_YZ)^v+(\ga^s_YD_XZ)^v-(\ga^s_Y\ga^a_XZ)^v\\
			&=-(D_X(\ga^s)(Y,Z))^v+([\ga^a_X,\ga^s_Y]Z)^v-(\ga^s_{\ga_XY}Z)^v,\\
			R^C(X^h,Y^v)Z^v&=\na^C_{(\na_XY)^v}Z^v+\na^C_{X^h}(\ga^s_YZ)^h
			+\na^C_{Y^v}(D_XZ)^v
			-\na^C_{Y^v}(\ga^a_XZ)^v\\
			&=(\ga^s_{\ga_XY}Z)^h-(\ga^s_{D_XY}Z)^h+(D_X\ga^s_YZ)^h-(\ga^a_X\ga^s_YZ)^h-(\ga^s_YD_XZ)^h+(\ga^s_Y\ga^a_XZ)^h,\\
			&=(D_X(\ga^s)(Y,Z))^v-([\ga^a_X,\ga^s_Y]Z)^v+(\ga^s_{\ga_XY}Z)^v,\\
				\end{align*}
			Now, by using \eqref{cu1} and \eqref{cu}	
			\begin{align*}
			R^C(X^h,Y^h)Z^h&=(D_Y(\ga)(X,Z))^h-(D_X(\ga)(Y,Z))^h+([\ga_X,\ga_Y]Z)^h+ +(D_X(\ga^a)(Y,Z))^h-(D_Y(\ga^a)(X,Z))^h+([\ga^a_Y,\ga^a_X]Z)^h\\
			&=\frac12D_Y(\ga+\ga^*)(X,Z)-\frac12D_X(\ga+\ga^*)(Y,Z)+([\ga_X,\ga_Y]Z)^h
			+([\ga^a_Y,\ga^a_X]Z)^h\\
			&=\frac12([\ga_X^*,\ga_Y^*]Z)^h-\frac12([\ga_X,\ga_Y]Z)^h+([\ga_X,\ga_Y]Z)^h
			+\frac14([\ga_Y-\ga_Y^*,\ga_X-\ga_X^*]Z)^h\\
			&=\frac12([\ga_X^*,\ga_Y^*]Z)^h+\frac12([\ga_X,\ga_Y]Z)^h
			+\frac14([\ga_Y-\ga_Y^*,\ga_X-\ga_X^*]Z)^h\\
			&=\frac14([\ga_X^*,\ga_Y^*]Z)^h+\frac14([\ga_X,\ga_Y]Z)^h
			-\frac14([\ga_Y,\ga_X^*]Z+[\ga_Y^*,\ga_X]Z)^h\\
			&=[\ga_X^s,\ga_Y^s].
			\end{align*}	
					Let $(E_1,\ldots,E_n)$ be a local orthonormal frame of $\prs$. We have obviously, $\rho^C(X^h,Y^h)=\rho^C(X^v,Y^v)=0$.
			 Then
	\begin{align*}2\rho^C(X^h,Y^v)&=\sum_{i=1}^n\left(g_1(R^C(X^h,Y^v)E_i^h,JE_i^h)+
	g_1(R^C(X^h,Y^v)E_i^v,JE_i^v)\right)\\
	&=\sum_{i=1}^n\left(g_1(R^C(X^h,Y^v)E_i^h,E_i^v)-
	g_1(R^C(X^h,Y^v)E_i^v,E_i^h)\right)\\
	&=2\sum_{i=1}^n\left( \langle-D_X(\ga^s)(Y,E_i)+[\ga^a_X,\ga^s_Y]E_i-\ga^s_{\ga_XY}E_i,E_i   \rangle \circ\pi_1    \right)\\
	&=-2\tr(\ga_{\ga_XY})\circ\pi_1-\sum_{i=1}^n\langle D_X(\ga)(Y,E_i),E_i\rangle\circ\pi_1-
	\sum_{i=1}^n\langle D_X(\ga^*)(Y,E_i),E_i\rangle\circ\pi_1\\
	&\stackrel{\eqref{nu}-\eqref{nam} }=-2\langle\ga_{X}Y,\tr_{\prs}(\ga^*)\rangle\circ\pi_1
	-2\langle\tr_{\prs}(D_X(\ga^*)),Y\rangle\circ\pi_1.
	\end{align*}
	\end{proof}
	
	The following proposition can be proved in a similar way as Proposition \ref{pryau}.
	\begin{pr}\label{prchern} Let $(M,\na,\prs)$ be an affine-Riemann manifold. If $\tr_{\prs}(\ga^*)=0$ then, for any $X\in\Ga(TM)$, $\tr_{\prs}(D_X(\ga^*))=0$. In particular, if $\tr_{\prs}(\ga^*)=0$ then the Ricci form $\rho^C$ of $(TM,J_1,g_1)$ vanishes.
		
	\end{pr}

	\section{The canonical sequence of Hermitian structures of  an affine-Riemann manifold: global and local results}\label{section4}
	
	In this section, we prove a global result on the sequence $(T^kM,J_k,g_k)$,  we give in local affine coordinates the necessary and sufficient conditions for $(TM,J_1,g_1)$ to be balanced or pluriclosed. We illustrate these results by many examples and we give some properties of affine-Riemann manifolds for which $(TM,J_1,g_1)$ is Vaisman.
	
	Let us start with the following  result which constitutes one of the main results of this paper.
	\begin{theo}\label{infinite} Let $(M,\na,\prs)$ be an affine-Riemann manifold. Then:
		\begin{enumerate}\item If $\ga=0$ then, for any $k\geq1$,  $\na^k$ is the Levi-Civita  connection of $g_k$ and $(T^kM,J_k,g_k)$ is K\"ahler flat.
			\item For some  $k\geq2$,  $(T^kM,J_k,g_k)$ is K\"ahler if and only if $\ga=0$.
			\item For some  $k\geq1$, $(T^kM,J_k,g_k)$ is locally conformally balanced if and only if $ (TM,J_1,g_1)$ is locally conformally balanced and this is equivalent to $d\xi=0$.
			
			\item For  $k_0\geq1$, $(T^{k_0}M,J_{k_0},g_{k_0})$ is balanced  if and only if 
			\begin{equation}\label{eqinf} \tr_{\prs}(\ga)=(2^{k_0}-1)\tr_{\prs}(\ga^*). \end{equation}
			In this case, $(T^{k_0+1}M,J_{k_0+1},g_{k_0+1})$ is Calabi-Yau with torsion and for any $k\not=k_0$, $(T^kM,J_k,g_k)$ is locally conformally balanced.
			\item If $\tr_{\prs}(\ga)=\tr_{\prs}(\ga^*)=0$ then, for any $k\geq1$, $(T^kM,J_k,g_k)$ is balanced, Calabi-Yau with torsion and its Chern Ricci form vanishes. 
			
			\end{enumerate}
		
		\end{theo}
		
		\begin{proof}\begin{enumerate}\item Note first that $\ga=0$ if and only if $D=\na$. On the other hand, it is obvious from Proposition \ref{Gaa} that $\ga=0$ if and only if the operator difference $\Ga$ for $(TM,\na^1,g_1)$ vanishes. By induction, we get that $\ga=0$ if and only if, for any $k\geq1$, the operator difference $\Ga^k$ of $(T^kM,\na^k,g_k)$ vanishes and the result follows.
				\item According to Corollary \ref{coh}, $(T^2M,J_2,g_2)$ is K\"ahler if and only if $\Ga^*=\Ga$. But from Proposition \ref{Gaa}, this implies that for any $X,Y\in\Ga(TM)$
				\[ \Ga_{X^h}Y^h=(\ga_XY)^h=\Ga_{X^h}^*Y^h=(\ga_X^*Y)^h\esp
				\Ga_{X^v}Y^v=-\frac12(\ga_X^*Y+\ga_Y^*X)^h=\Ga_{X^v}^*Y^v=\frac12(\ga_X^*Y+\ga_Y^*X)^h  \]and hence $\ga=0$. By induction, we get the result.
				\item Fix $k\geq1$. Then $(T^kM,J_k,g_k)$ is locally conformally balanced if and only if its Lee form $\theta_k$ is closed. But from Proposition \ref{Gaa}, $\theta_k=\pi_k^*\circ\ldots\circ \pi_1^*((2^k-1)\al-\xi)$. The first Koszul form $\al$ being closed we get the result.
			\item	Fix $k_0\geq1$. Then $(T^{k_0}M,J_{k_0},g_{k_0})$ is  balanced if and only if  $\theta_k=\pi_k^*\circ\ldots\circ \pi_1^*((2^k-1)\al-\xi)=0$ which is equivalent to $\tr_{\prs}(\ga)=(2^{k_0}-1)\tr_{\prs}(\ga^*)$. But from Proposition \ref{Gaa}, $\theta_k=-\xi_k$ and one can use Proposition \ref{pryau} to deduce that if $(T^{k_0}M,J_{k_0},g_{k_0})$ is  balanced then $(T^{k_0+1}M,J_{k_0+1},g_{k_0+1})$ is Calabi-Yau with torsion. On the other hand, since $\al$ is closed then $d\xi=0$ which completes the proof.
			\item It is a consequence of what above and Propositions 
			\ref{pryau}-\ref{prchern}.
			\qedhere

				\end{enumerate}

			\end{proof}

					\begin{exem}\label{exem1}\begin{enumerate}
							\item In the item 4 of Theorem \ref{infinite}, one can build a balanced Hermitian structure on $T^kM$ by solving an equation on $(M,\na,\prs)$. Let us give some examples of this situation.
							\begin{enumerate}\item We consider the left symmetric product on $\R^3$ given by
							\[ e_1\bullet e_1=ae_1,e_1\bullet e_2=ae_2+e_3,e_1\bullet e_3=e_2+ae_3,e_2\bullet e_1=ae_2,e_3\bullet e_1=ae_3. \]
							The associated non vanishing Lie brackets are given by
							\[ [e_1,e_2]=e_3,\; [e_1,e_3]=e_2. \]
							We denote by $G$ the connected simply-connected Lie group associated to $(\R^3,\br)$ and by $\na$ the left invariant flat torsionless connection on $G$ defined by $\bullet$.
							 For $a=1$,  the left invariant metric on $G$ associated to the scalar product $\left[ \begin {array}{ccc} 1&1&0\\ \noalign{\medskip}1&3&1
								\\ \noalign{\medskip}0&1&1\end {array} \right]$ on $\R^3$ satisfies \eqref{eqinf} for $k_0=2$ with $\tr_{\prs}(\ga)\not=0$ and $\tr_{\prs}(\ga^*)\not=0$. Thus $(T^2G,J_2,g_2)$ is balanced, $(T^3G,J_3,g_3)$ is Calabi-Yau with torsion  and, for any $k\not=2$, $(T^kG,J_k,g_k)$ is locally conformally balanced not balanced.
								\item We consider the left symmetric product on $\R^3$ given by
								\[ e_1\bullet e_1=ae_1,e_1\bullet e_2=(1+a)e_2,e_1\bullet e_3=(1+a)e_3,e_2\bullet e_1=ae_2,e_3\bullet e_1=ae_3. \]
								The associated non vanishing Lie brackets are given by
								\[ [e_1,e_2]=e_2,\; [e_1,e_3]=e_3. \]
								We denote by $G$ the connected simply-connected Lie group associated to $(\R^3,\br)$ and by $\na$ the left invariant flat torsionless connection on $G$ defined by $\bullet$.
								For $a=\frac83$,  the left invariant metric on $G$ associated to the scalar product $\left[ \begin {array}{ccc} \la&0&0\\ \noalign{\medskip}0&\mu&0								\\ \noalign{\medskip}0&0&\nu\end {array} \right]$ on $\R^3$ satisfies \eqref{eqinf} for $k_0=4$ with $\tr_{\prs}(\ga)\not=0$ and $\tr_{\prs}(\ga^*)\not=0$. Thus $(T^4G,J_4,g_4)$ is balanced, $(T^5G,J_5,g_5)$ is Calabi-Yau with torsion and, for any $k\not=4$, $(T^kG,J_k,g_k)$ is locally conformally balanced not balanced.
								
								\end{enumerate}

							\end{enumerate}

						\end{exem}

	Let us compute the Koszul forms of an affine-Riemann manifold in local affine coordinates.
	
	\begin{pr}\label{ltheta} Let $(M,\na,\prs)$ be an affine-Riemann manifold. For 
		 any  system of affine coordinates $(x_1,\ldots,x_n)$,
		 \begin{equation} 
		 \label{theta0} \alpha=\frac12d\ln(\det G)\esp 
		 \xi=\sum_{j=1}^n\left( \sum_{h,k}\mu^{kh}\frac{\partial \mu_{jh}}{\partial x_k}    \right)dx_j-\al,
		 \end{equation}		 
		  where $\mu_{hk}=\langle\partial_{x_h},\partial_{x_k}\rangle$ and the matrix $(\mu^{hk})_{1\leq h,k\leq n}=G^{-1}$ where $G=(\mu_{hk})_{1\leq h,k\leq n}$. 
		
	\end{pr}
	\begin{proof} Let $(x_1,\ldots,x_n)$ be a system of affine coordinates. The Riemannian volume $\mu$ is given by 
		\[ \mu=\sqrt{\det G}dx_1\wedge\ldots\wedge dx_n \] and the first formula is a consequence of the relation $\na_X\mu=\al(X)\mu$ for any $X\in\Ga(TM)$ (see Proposition \ref{pr1}).

		 On the other hand, we consider a local orthonormal frame $(E_1,\ldots,E_n)$ of $\prs$ and we denote by $P=(p_{ij})_{1\leq i,j\leq k}$ the passage matrix from $(\partial_{x_1},\ldots,\partial_{x_n})$ to $(E_1,\ldots,E_n)$. We have $P^tGP=\mathrm{I_n}$. For any $j=1,\ldots,n$,
		\begin{align*}
		\al(\partial_{x_j})&=\sum_{i=1}^n \langle \ga_{\partial_{x_j}}E_i,E_i\rangle\\
		&=\sum_{i,h,k}p_{ki}p_{hi} \langle D_{\partial_{x_j}}\partial_{x_k},\partial_{x_h}\rangle\\
		&=\frac12\sum_{h,k}m_{kh} \left(\frac{\partial \mu_{hk}}{\partial x_j}+\frac{\partial \mu_{jh}}{\partial x_k}-\frac{\partial \mu_{jk}}{\partial x_h}\right).\end{align*} In the same way,\begin{align*} 																					\xi(\partial_{x_j})&=\sum_{i=1}^n \langle \ga_{E_i}E_i,\partial_{x_j}\rangle\\
		&=\sum_{i,h,k}p_{ki}p_{hi} \langle D_{\partial_{x_h}}\partial_{x_k},\partial_{x_j}\rangle\\
		&=\frac12\sum_{h,k}m_{kh} \left(\frac{\partial \mu_{kj}}{\partial x_h}+\frac{\partial \mu_{jh}}{\partial x_k}-\frac{\partial \mu_{hk}}{\partial x_j}\right)	\\
		&=\sum_{h,k}m_{kh}\frac{\partial \mu_{jh}}{\partial x_k}-\al(\partial_{x_j}),																			\end{align*}	
			where $m_{kh}=\sum_{i=1}^np_{ki}p_{hi}$. But $(m_{kh})_{1\leq k,h\leq n}=PP^t$ and the result follows from the formula $P^tGP=\mathrm{I}_n$.	
	\end{proof}
	
	\begin{exem}\label{exem2} Let $f:\R^2\too \R$ be a smooth function. Consider the affine-Riemann manifold $(\R^2,\na^0,\prs)$ where $\na^0$ is the canonical connection of $\R^2$ and $$\prs=\left(\begin{array}{cc}\cosh(f(x,y))&\sinh(f(x,y))\\
		\sinh(f(x,y))&\cosh(f(x,y)) \end{array}   \right).$$ Then $\det\prs=1$ and, by virtue of Proposition \ref{ltheta}, $\al=0$. According to Proposition \ref{prchern}, the Chern Ricci form of $(T\R^2,J_1,g_1)$ vanishes. 
		
		\end{exem}
	
	The following theorem gives a large class of balanced metrics non-K\"ahler on $\mathbb{C}^2$ endowed with its canonical complex structure.
	\begin{theo}\label{dim2} We consider $M=\R^2$ endowed its canonical affine structure and $\prs=\left(\begin{array}{cc}\mu_{11}&\mu_{12}\\\mu_{12}&\mu_{22}\end{array}    \right)$ a Riemannian metric. Then $(TM,J_1,g_1)$ is balanced if and only if there exists smooth functions $\nu:\R^2\too\R$, $f,h:\R\too\R$ such that
		\[ \mu_{12}=\nu,\; \mu_{11}(x_1,x_2)=f(x_1)+\int\frac{\partial\nu}{\partial x_1}(x_1,x_2)dx_2\esp \mu_{22}(x_1,x_2)=h(x_2)+\int\frac{\partial\nu}{\partial x_2}(x_1,x_2)dx_1. \]
		
	\end{theo}
	
	\begin{proof} According to  Proposition \ref{ltheta}, 
		\begin{align*}
		(\xi-\al)(\partial_{x_1})&=\mu^{11}\frac{\partial\mu_{11}}{\partial x_1}+\mu^{12}\left(\frac{\partial\mu_{11}}{\partial x_2}+\frac{\partial\mu_{12}}{\partial x_1} \right)+\mu^{22}\frac{\partial\mu_{12}}{\partial x_2}-\frac1{\det G}\frac{\partial (\mu_{11}\mu_{22}-\mu_{12}^2)}{\partial x_1}\\
		&=\frac1{\det G}\left(\mu_{22}\frac{\partial\mu_{11}}{\partial x_1}-\mu_{12}\left(\frac{\partial\mu_{11}}{\partial x_2}+\frac{\partial\mu_{12}}{\partial x_1} \right)+\mu_{11}\frac{\partial\mu_{12}}{\partial x_2} \right)
		-\frac1{\det G}\frac{\partial (\mu_{11}\mu_{22}-\mu_{12}^2)}{\partial x_1}\\
		&=\frac1{\det G}\left(-\mu_{12}\left(\frac{\partial\mu_{11}}{\partial x_2}-\frac{\partial\mu_{12}}{\partial x_1}  \right) +
		\mu_{11}\left(\frac{\partial\mu_{12}}{\partial x_2}-\frac{\partial\mu_{22}}{\partial x_1}  \right)      \right),\\
		(\xi-\al)(\partial_{x_2})&=\mu^{11}\frac{\partial\mu_{12}}{\partial x_1}+\mu^{12}\left(\frac{\partial\mu_{12}}{\partial x_2}+\frac{\partial\mu_{22}}{\partial x_1} \right)+\mu^{22}\frac{\partial\mu_{22}}{\partial x_2}-\frac1{\det G}\frac{\partial (\mu_{11}\mu_{22}-\mu_{12}^2)}{\partial x_2}\\
		&=\frac1{\det G}\left(\mu_{22}\frac{\partial\mu_{12}}{\partial x_1}-\mu_{12}\left(\frac{\partial\mu_{12}}{\partial x_2}+\frac{\partial\mu_{22}}{\partial x_1} \right)+\mu_{11}\frac{\partial\mu_{22}}{\partial x_2}    \right)
		-\frac1{\det G}\frac{\partial (\mu_{11}\mu_{22}-\mu_{12}^2)}{\partial x_2}\\
		&=\frac1{\det G}\left(-\mu_{22}\left(\frac{\partial\mu_{11}}{\partial x_2}-\frac{\partial\mu_{12}}{\partial x_1}  \right) +
		\mu_{12}\left(\frac{\partial\mu_{12}}{\partial x_2}-\frac{\partial\mu_{22}}{\partial x_1}  \right)      \right).
		\end{align*}Thus, the vanishing of $\xi-\al$ is equivalent to
		\[ \frac{\partial \mu_{11}}{\partial x_2}
		-\frac{\partial \mu_{12}}{\partial x_1}= \frac{\partial \mu_{12}}{\partial x_2}
		-\frac{\partial \mu_{22}}{\partial x_1}=0\]
		 and we get the desired result.
			\end{proof}
	\begin{exem}\label{exem3} For any smooth functions $f,h:\R\too\R$, the metric $$\prs=\left(\begin{array}{cc} e^{x+y}+e^{f(x)}&e^{x+y}\\e^{x+y}&e^{x+y}+e^{h(y)}    \end{array} \right)$$ satisfies the condition of the last corollary and hence defines a balanced Hermitian metric on $\mathbb{C}^2$.

		\end{exem}

	The following theorem  gives a large class of balanced metrics non-K\"ahler and also Calabi-Yau metrics on $\mathbb{C}^m$ endowed with its canonical complex structure. 
	\begin{theo}\label{dimn} We consider $M=\R^n$ endowed its canonical affine structure and $\prs=\mathrm{Diag}(\mu_1,\ldots,\mu_n)$ a Riemannian metric. For $k_0\geq1$,  $(T^{k_0}M,J_{k_0},g_{k_0})$ is balanced if and only if there exists $(f_1,\ldots,f_n)$ a family of positive functions  such that, for $j=1,\ldots,n$,
		\[ \frac{\partial f_j}{\partial x_j}=0\esp \mu_j=\frac{f_1\ldots f_n}{f_j^{(n2^{k_0-1}-1)}}. \]
		In this case, $(T^{k_0+1}M,J_{k_0+1},g_{k_0+1})$ is Calabi-Yau with torsion and,  for any $k\not=k_0$, $(T^kM,J_k,g_k)$ is locally conformally balanced.
	\end{theo}
	
	\begin{proof} Note first that, by virtue of \eqref{theta0}, for any $k\geq1$,
		\begin{equation}
		\xi-(2^k-1)\al=-\sum_{j=1}^n\frac{\partial(\ln(\rho_j))}{\partial x_j}dx_j\esp \rho_j=\frac{(\mu_1.\ldots\mu_n)^{2^{k-1}}}{\mu_j},
		\end{equation}
		and hence, according to Theorem \ref{infinite}, 
		 $(T^{k_0}M,J_{k_0},g_{k_0})$ is balanced if and only if, for $j=1,\ldots,n$,
		\[ \frac{\partial\rho_j}{\partial x_j}=0\esp \rho_j=\frac{(\mu_1\ldots\mu_n)^{2^{k_0-1}}}{\mu_j}. \]
		We have obviously that $\rho_1\ldots\rho_n=(\mu_1\ldots\mu_n)^{n2^{k_0-1}-1}$ and hence
		\[ \mu_j=\frac{(\rho_1\ldots\rho_n)^{\frac1{n2^{k_0-1}-1}}}{\rho_j}. \]If we put $f_j=(\rho_j)^{\frac1{n2^{k_0-1}-1}}$, we get the desired result and Theorem \ref{infinite} permits to conclude.
			\end{proof}
			
			Now, we give the conditions in local coordinates so that $(TM,J_1,g_1)$ is  pluriclosed.
			
				\begin{theo}\label{pluril} Let $(M,\na,\prs)$ be an affine-Riemann manifold. Then $(TM,J_1,g_1)$ is pluriclosed if and only if, for any affine coordinates $(x_1,\ldots,x_n)$,
					\begin{equation}\label{pluri}
					\frac{\partial^2 \mu_{ik}}{\partial x_j\partial x_h}+\frac{\partial^2 \mu_{jh}}{\partial x_i\partial x_k}=
					\frac{\partial^2 \mu_{jk}}{\partial x_i\partial x_h}
					+\frac{\partial^2 \mu_{ih}}{\partial x_j\partial x_k},
					\end{equation}for any $1\leq i<j\leq n$ and $1\leq k<h\leq n$ and where $\mu_{ij}=\langle\partial_{x_i},\partial_{x_j}\rangle$. When $\dim M=2$, \eqref{pluri} reduces to
					\[ \frac{\partial^2 \mu_{11}}{\partial x_2^2}+\frac{\partial^2 \mu_{22}}{\partial x_1^2}=2
					\frac{\partial^2 \mu_{12}}{\partial x_1\partial x_2}. \]
					
				\end{theo}
				\begin{proof} According to Proposition \ref{prdc}, $(TM,J_1,g_1)$ is pluriclosed if and only if, for any $1\leq i<j\leq n$ and any $1\leq k<h\leq n$,
					\begin{align*}
					0&=\langle K(\partial_{x_i},\partial_{x_j})\partial_{x_k},\partial_{x_h}\rangle
					-\langle\ga_{\partial_{x_j}}\partial_{x_k},\ga_{\partial_{x_i}}
					\partial_{x_h}\rangle+\langle\ga_{\partial_{x_i}}\partial_{x_k},\ga_{\partial_{x_j}}
					\partial_{x_h}\rangle\\
					&=-\langle D_{\partial_{x_i}}D_{\partial_{x_j}}\partial_{x_k},\partial_{x_h}\rangle
					+\langle D_{\partial_{x_j}}D_{\partial_{x_i}}\partial_{x_k},\partial_{x_h}\rangle
					-\langle D_{\partial_{x_j}}\partial_{x_k},D_{\partial_{x_i}}
					\partial_{x_h}\rangle+\langle D_{\partial_{x_i}}\partial_{x_k},D_{\partial_{x_j}}
					\partial_{x_h}\rangle\\
					&=-\partial_{x_i}.\langle D_{\partial_{x_j}}\partial_{x_k},\partial_{x_h}\rangle
					+\partial_{x_j}.\langle D_{\partial_{x_i}}\partial_{x_k},\partial_{x_h}\rangle\\
					&=-\frac12\left(\frac{\partial^2 \mu_{kh}}{\partial x_i\partial x_j}
					+\frac{\partial^2 \mu_{jh}}{\partial x_i\partial x_k}
					-\frac{\partial^2 \mu_{kj}}{\partial x_i\partial x_h}
					-\frac{\partial^2 \mu_{kh}}{\partial x_i\partial x_j}
					-\frac{\partial^2 \mu_{ih}}{\partial x_j\partial x_k}
					+\frac{\partial^2 \mu_{ki}}{\partial x_j\partial x_h}\right)\\
					&=-\frac12\left(
					\frac{\partial^2 \mu_{jh}}{\partial x_i\partial x_k}
					-\frac{\partial^2 \mu_{kj}}{\partial x_i\partial x_h}
					-\frac{\partial^2 \mu_{ih}}{\partial x_j\partial x_k}
					+\frac{\partial^2 \mu_{ki}}{\partial x_j\partial x_h}\right).
					\end{align*}\end{proof}
				\begin{co}\label{cop1} We consider $M=\R^n$ endowed with its canonical affine structure and $\prs=\mathrm{Diag}(\mu_1,\ldots,\mu_n)$. Then $(TM,J_1,g_1)$ is pluriclosed if and only if, for any $i\not=j$,  $h\not=j$ and $h\not=i$,
					\begin{equation} \label{plurid}\frac{\partial^2 \mu_{i}}{\partial x_j^2}+\frac{\partial^2 \mu_{j}}{\partial x_i^2}=0\esp \frac{\partial^2 \mu_{i}}{\partial x_j\partial x_h}=
					0. \end{equation}
					In particular, if we take $\mu_i=e^{f_1(x_i)}$ then $(TM,J_1,g_1)$ is pluriclosed.
				\end{co}

					We end this section by giving some properties of affine-Riemann manifolds with Vaisman tangent bundle.
					
					\begin{pr}Let $(M,\na,\prs)$ be an affine-Riemann manifold such that $(TM,J_1,g_1)$ is Vaisman. Then the following assertions hold.
						\begin{enumerate}\item If $(TM,J_1,g_1)$ is non-K\"ahler then
						  the vector field $\Pi=\tr(\ga^*)-\tr(\ga)$ is a non-vanishing parallel vector field with respect $\na$ an $D$ and the distribution $\Pi^\perp$ is integrable and defines a codimension one totally geodesic foliation on $M$.
						\item If $\dim M=2$ and $(TM,J_1,g_1)$ is  non-K\"ahler
						 then the curvature of $\prs$ vanishes.
						\end{enumerate}
					\end{pr}
					
					\begin{proof} According to Proposition \ref{vaisman}, $(TM,J_1,g_1)$ is Vaisman if and only if \eqref{LCK} holds and the vector field $\Pi:=\tr_{\prs}(\ga^*-\ga)$ is parallel with respect to both $D$ and $\na$.
						\begin{enumerate}\item One can deduce from \eqref{LCK} that if $(TM,J_1,g_1)$ is non-K\"ahler then $\Pi$ is non zero parallel. For any vector fields $X,Y$ orthogonal to $\Pi$, $D_XY$ is also orthogonal to $\Pi$ and hence $F=\Pi^\perp$ is integrable and defines a totally geodesic foliation.
							\item It is obvious that a Riemannian surface with a non zero parallel vector field is flat.\qedhere 
							
							\end{enumerate}
						
						\end{proof}
						
						The geometry of Riemannian manifolds endowed with a codimension one totally geodesic foliation is well-understood (see \cite{ghys}). It is then an interesting problem to study affine-Riemann manifolds with Vaisman tangent bundle.

					\begin{exem}\label{exem4} Consider $\R^2$  endowed with a flat Riemannian metric $\prs=\left(
						\begin{array}{cc}a&b\\b&c\end{array} \right),a,b,c\in\R$ and the flat torsionless connection given by
						\[ \na_{\partial x}\partial_x=\partial_y\esp \na_{\partial x}\partial_y=\na_{\partial y}\partial_y=0. \]
						Then
						\[ \ga_{\partial_x}^*=\left( \begin {array}{cc} {\frac {cb}{ac-{b}^{2}}}&{\frac {{c}^{2}}{a
								c-{b}^{2}}}\\ \noalign{\medskip}-{\frac {{b}^{2}}{ac-{b}^{2}}}&-{
							\frac {cb}{ac-{b}^{2}}}\end {array} \right) 
						\esp\ga_{\partial_y}^*=0. \]
						Consider the orthonormal frame $(E_1,E_2)$ given by
						\[ E_1=\frac1{\sqrt{a}}\partial_{x}\esp E_2=\frac1{\sqrt{a(ca-b^2)}} \left(a\partial_{y}-b\partial_{x}\right). \]
						By using this orthonormal frame, one can check easily that
						\[ \tr_{\prs}(\ga^*)=0\esp \tr_{\prs}(\ga)=\frac{c}{ac-b^2}\partial_y. \]One can also check that
						\[ \ga_{\partial_x}^*\partial_y-\ga_{\partial_y}^*\partial_x=
						\langle \tr_{\prs}(\ga^*)-\tr_{\prs}(\ga),\partial_x\rangle
						\partial_y-\langle \tr_{\prs}(\ga^*)-\tr_{\prs}(\ga),\partial_y\rangle
						\partial_x \]and hence \eqref{LCK} holds. Thus $(T\R^2,J_1,g_1)$ is Vaisman and, actually this structure is   left invariant.
					\end{exem}

		\section{Rigid affine-Riemann manifolds}\label{section5}
		
In this section, we study		the case where  $(M,\na,\prs)$ is an affine-Riemann manifold satisfying
		\begin{equation}\label{dga} D(\ga)=0. \end{equation}This condition implies that $D\al=D\xi=0$ and in particular $d\xi=0$. According to Theorem \ref{infinite}, we get the following result which justifies the study of this class of affine-Riemann manifolds.
		
		\begin{pr} \label{thedga} Let $(M,\na,\prs)$ be an affine-Riemann satisfying \eqref{dga}. Then, for any $k\geq1$, $(T^kM,J_k,g_k)$ is locally conformally balanced.
			
			\end{pr}
		Affine-Riemann manifolds satisfying \eqref{dga} will be called {\it rigid}.	
		
		\begin{pr}\label{rigidpr} Let $(M,\na,\prs)$ be a Riemannian manifold endowed with a torsionless connection. Then the following assertions are equivalent:
			\begin{enumerate}\item $(M,\na,\prs)$ is a rigid affine-Riemann manifold.
				\item The difference tensor $\ga=D-\na$ satisfies:
				\[ D(\ga)=0\esp K(X,Y)=[\ga_X,\ga_y] \]for any $X,Y\in\Ga(TM)$.
				
				\end{enumerate}
			
			\end{pr}
			
			\begin{proof} It is a consequence of the following formula
				\[ K(X,Y)Z=R^\na(X,Y)Z+D_Y(\ga)(X,Z)-D_X(\ga)(Y,Z)+[\ga_X,\ga_Y](Z), \]
				where $R^\na$ is the curvature of $\na$.
					\end{proof}
		
		Le us show that we can apply the following theorem du to Kostant \cite[Theorem 4]{kostant} to get an interesting description of  rigid affine-Riemann manifolds with complete Riemannian metric.
		\begin{theo}[Kostant]\label{kostant} Let $A$ be a connection on a simply-connected manifold $M$. Assume that there exists a second  connection $B$ on $M$ such that
			\begin{enumerate}	
				\item $B$ is invariant under parallelism, i.e., $BT=0$ and $BR=0$ where $T$ and $R$ are, respectively, the torsion and the curvature of $B$. \item $A$ is rigid with respect to $B$, i.e., $S=B-A$ is $B$-parallel. \item M is complete with respect to $B$.
			\end{enumerate}	 Let $\G$ be the Lie algebra of infinitesimal $B$ affine transformations $X$ on M such that $(L_X-B_X)p\in \mathfrak{s}_p$ for some (and hence every) point $p\in M$ where $L_X$ is the Lie derivative in the direction of $X$ and $\mathfrak{s}_p$ is the $B$-holonomy algebra at $p$. Then the infinitesimal action $\rho:\G\too\Ga(TM)$ integrates to an action $\phi:G\too\mathrm{Diff}(M)$ of a simply-connected Lie group $G$ which preserves both $A$ and $B$. Moreover,
			 $M$ is a reductive homogeneous space with respect to the action of $G$.

		\end{theo}
		
		Let us see that the condition \eqref{dga} implies the hypothesis of  Theorem \ref{kostant} for $A=\na$ and $B=D$. Indeed, the condition \eqref{dga} is equivalent to $\na$ is rigid with respect to $D$ and from \eqref{K} 
		\[ K(X,Y)=[\ga_X,\ga_Y]\quad\mbox{for}\; X,Y\in\Ga(TM). \]Since $\ga$ is parallel, we get that
		 $DK=0$ and hence $D$ is invariant under parallelism. If we suppose that $M$ is simply-connected and $D$ is complete we can apply Theorem \ref{kostant} and get the following result. Note that since $K$ is parallel the Lie algebra of holonomy is given by
		 \[ \mathfrak{s}_p=\left\{\sum K(u_i,v_i), u_i,v_i\in T_pM \right\}. \]
		 Moreover, since $D$ is torsionless $L_X-B_X=-DX$ for any $X\in\Ga(TM)$.
		 \begin{theo}\label{rigid} Let $(M,\na,\prs)$ be a simply-connected rigid affine-Riemann manifold such that $\prs$ is complete.  Consider
		 	\[ \G=\left\{X\in\Ga(TM),DX=\sum K(U_i,V_i)   \right\}. \]
		 	Then the action $\rho:\G\too \Ga(TM)$ integrates to an action $\phi:G\too\mathrm{Diff}(M)$ of a simply-connected Lie group which preserves both $\na$ and $\prs$. Moreover, $M$ is homogeneous reductive under this action and $(M,\prs)$ is a Riemannian symmetric space.
		 	
		 	\end{theo}

		 	The following result shows that there is a correspondence between simply-connected complete flat rigid affine-Riemann manifolds and associative commutative algebras.
		 	\begin{co}\label{co2} Let $(M,\na,\prs)$ be simply-connected  rigid affine-Riemann manifold with $K=0$ and $\prs$ is complete. Then $(M,\prs)$ is isometric to $\R^n$ with its canonical metric and there exists an associative commutative product $\bullet$ on $\R^n$ such that $\ga_uv=u\bullet v$ for any $u,v\in T_p\R^n$ and $p\in\R^n$.
		 		
		 		Conversely,	let $(A,\bullet)$ be a real finite dimensional associative commutative algebra and $\prs$ a scalar product on $A$. The product $\bullet$ defines on $A$ a flat torsionless connection $\na$ and  $(A,\na,\prs)$ is a rigid affine-Riemann manifold.

		 	\end{co}
		 	
		 	\begin{proof} Since $K=0$ and $\prs$ is complete  then $(M,\prs)$ is isometric to $\R^n$ with its canonical metric. Moreover, $\ga$ is parallel and hence it is given by $\ga_uv=u\bullet v$ for any $u,v\in T_p\R^n$ and $p\in\R^n$, where $\bullet$ is a commutative product on $\R^n$. The relation $K(u,v)=[\ga_u,\ga_v]=0$ implies that $\bullet$ is associative.
		 		The converse is obvious.
		 		 		\end{proof}
		 		 		
		 		 Theorem \ref{rigid} suggests us to look for rigid affine-Riemann manifolds among symmetric spaces and we give now a practical method to achieve this task.
		 		 	The following result is a consequence of Proposition \ref{rigidpr} and the properties of the holonomy representation of symmetric spaces (see \cite[ Proposition 10.79]{besse}).	
		 		\begin{pr}\label{sym} Let $(M,\prs)$ be a simply-connected Riemannian symmetric space and $G$ its group of isometries, $o$ a fixed point of $M$ and $\ga^0:T_oM\times T_oM\too T_oM$ a symmetric product such that
		 			\[ \ad(a)(\ga_u^0v)=\ga_{\ad(a)u}^0v+\ga_{u}^0\ad(a)v\esp K(u,v)=[\ga_u^0,\ga_v^0]\quad u,v\in T_oM, a\in \G_0, \]where $\G_0$ is the Lie algebra of the isotropy at $o$,  $\ad:\G_0\too \mathrm{End}(T_oM)$ is the infinitesimal isotropy representation and $K$ is the curvature of $\prs$. Then $\ga^0$ is invariant by holonomy and defines a parallel tensor $\ga$ on $M$. Moreover, $(M,D-\ga,\prs)$ is a rigid affine-Riemann manifold. 
		 			
		 			\end{pr} 
		 			
		 			We illustrate this proposition by the following example.

		 	\begin{exem}\label{exem5} Consider $M := \mathrm{SPD}(n)$ the set of real symmetric positive definite $n \times n$ matrices, which is an open subset of $\mathrm{S}(n)$: the vector space of real symmetric $n \times n$ matrices. The connected Lie group $G := \mathrm{GL}^+(n, \R)$ of positive determinant $n \times n$ matrices acts transitively on $M$ : $g · x := gxg^t$  , and the isotropy subgroup at $\mathrm{I}_n$ is $H := \mathrm{SO}(n)$.
The Lie algebra of $H$ is $\h=\mathfrak{so}(n,\R)$ and with $\mathfrak{m} := \mathrm{S}(n)$, we have a canonical decomposition
$$\G=\h\oplus \mathfrak{m}\esp  \Ad(H)(\mathfrak{m})\subset\mathfrak{m}.$$
The scalar product on $\mathfrak{m}$ given by $\langle A,B\rangle_0=\tr(AB)$ is $\Ad(H)$ invariant and hence defines a $G$-invariant Riemannian metric $\prs$ on $M=G/H$ and $(G/H,\prs)$ is a symmetric space. Its curvature at $T_{\pi(e)}G/H=\mathfrak{m}$ is given by
\[ K(A,B)C=[[A,B],C],\quad A,B,C\in\mathfrak{m}. \]
On the other hand, the  product $\ga^0:\mathfrak{m}\times\mathfrak{m}\too \mathfrak{m}$
\[ \ga_A^0B=AB+BA \]satisfies
\[ K(A,B):=[\ga_A^0,\ga_B^0]\esp K(A,B)\ga_C^0E=\ga_{K(A,B)C}^0E+\ga_{C}^0K(A,B)E,\quad A,B,C,E\in\mathfrak{m}. \]
Since the holonomy Lie algebra of $(G/H,\prs)$ is generated by $K$ and $G/H$ is 
simply-connected and by using Proposition \ref{rigidpr}, one can see that $\ga^0$ defines an invariant parallel tensor field $\ga$ on $G/H$ such that, if $D$ is the Levi-Civita connection of $(G/H,\prs)$,  $(G/H,D-\ga,\prs)$ is a rigid affine-Riemann manifold. Moreover, one can see that $\ga=\ga^*$ and hence $(G/H,D-\ga,\prs)$ is a Hessian manifold.
		 		\end{exem}

		 	We determine now complete rigid affine-Riemann manifolds of dimension 2 and 3. We start with the following propositions.
		 	\begin{pr}\label{le} The  manifold $\R\times S^2(r)$   carries a family depending on a non null real parameter of affine structures $\na^c$ such that $(\R\times S^2(r),\na^c,\prs_0)$ is a rigid affine-Riemann manifold where $S^2(r)$ is the 2-sphere of radius $r$ and $\prs_0$ is the canonical metric of $\R\times S^2(r)$. Moreover, for $c=\pm\frac{\sqrt{2}}{r}$ we have $\tr_{\prs_0}(\ga)=0$ and hence $(T(\R\times S^2(r)),J_1,g_1)$ is Calabi-Yau with torsion.
		 		\end{pr}
		 		
		 	\begin{proof} By virtue of Proposition \ref{rigidpr}, a rigid affine-Riemann structure on $\R\times S^2(r)$ is  a symmetric tensor field $\ga$ of type $(2,1)$ such that
		 		\[ D(\ga)=0\esp K(X,Y)=[\ga_X,\ga_Y],\quad X,Y\in\Ga(T(\R\times S^2(r))), \] where $D$ and $K$ are, respectively, the Levi-Civita and the curvature of $\prs_0$. According to the holonomy principle (see \cite[pp. 282]{besse}), this is equivalent to the following: for a fixed
		 		  $p\in \R\times S^2(r)$, there exists $\ga^0:T_p(\R\times S^2(r))\times T_p(\R\times S^2(r))\too T_p(\R\times S^2(r))$ such that, for any $u,v\in T_p(\R\times S^2(r))$ and any $h\in\G_p$,
		 		\begin{equation}\label{ho} h.\ga_{u}^0v=\ga_{h.u}^0v+\ga_u^0(h.v)\esp K(u,v)=[\ga_u^0,\ga_u^0], \end{equation} where
		 		$\G_p$ is the holonomy Lie algebra at $p$. Take $p=(1,(0,0,1))$, denote by $(e_1,e_2,e_3)$ the canonical basis of $\R^3$ and $e_0$ the generator of $\R$. Then $T_p(\R\times S^2(r))=\mathrm{span}(e_0,e_1,e_2)$ and $$\G_p=\left\{\left(\begin{array}{ccc}0&0&0\\0&0&\la\\0&-\la&0
		 		\end{array}    \right),\la\in\R      \right\}.$$Put $\ga_{u}v=A_0(u,v)e_0+A_1(u,v)e_1+A_2(u,v)e_2$. Then the first equation in \eqref{ho} is equivalent to
		 		\[ \begin{cases}A_0(h.u,v)+A_0(u,h.v)=0,\\
		 		A_1(h.u,v)+A_1(u,h.v)=\la A_2(u,v),\\
		 		A_2(h.u,v)+A_2(u,h.v)=-\la A_1(u,v),
		 		\end{cases} \]for any $u,v\in T_p(\R\times S^2(r))$ and $h\in\G_p$. The solutions of this system of equations are given by their matrices in $(e_0,e_1,e_2)$
		 		\[ A_0=\left( \begin {array}{ccc} a_{{1,1}}&0&0\\ \noalign{\medskip}0&a_{{3,
		 				3}}&0\\ \noalign{\medskip}0&0&a_{{3,3}}\end {array} \right),
		 		A_1=\left( \begin {array}{ccc} 0&c_{{1,3}}&-c_{{1,2}}
		 		\\ \noalign{\medskip}c_{{1,3}}&0&0\\ \noalign{\medskip}-c_{{1,2}}&0&0
		 		\end {array} \right) 
		 		\esp A_2=\left( \begin {array}{ccc} 0&c_{{1,2}}&c_{{1,3}}\\ \noalign{\medskip}
		 		c_{{1,2}}&0&0\\ \noalign{\medskip}c_{{1,3}}&0&0\end {array} \right)
		 		 \]and hence
		 		 \[ \ga_{e_0}^0=\left( \begin {array}{ccc} a_{{1,1}}&0&0\\ \noalign{\medskip}0&c_{{1,
		 		 		3}}&-c_{{1,2}}\\ \noalign{\medskip}0&c_{{1,2}}&c_{{1,3}}\end {array}
		 		 \right) 
		 		 ,\ga_{e_1}^0=\left( \begin {array}{ccc} 0&a_{{3,3}}&0\\ \noalign{\medskip}c_{{1,3}
		 		 }&0&0\\ \noalign{\medskip}c_{{1,2}}&0&0\end {array} \right)
		 		 	\esp \ga_{e_2}^0=\left( \begin {array}{ccc} 0&0&a_{{3,3}}\\ \noalign{\medskip}-c_{{1,2
		 		 		}}&0&0\\ \noalign{\medskip}c_{{1,3}}&0&0\end {array} \right)
		 		 		.
		 		 	\]The second equation of \eqref{ho} is equivalent to
		 		 	 \begin{align*} K(e_0,e_1)&=\left( \begin {array}{ccc} 0&a_{{1,1}}a_{{3,3}}-c_{{1,3}}a_{{3,3}}&a_
		 		 	 {{3,3}}c_{{1,2}}\\ \noalign{\medskip}-a_{{1,1}}c_{{1,3}}-{c_{{1,2}}}^{
		 		 	 	2}+{c_{{1,3}}}^{2}&0&0\\ \noalign{\medskip}-a_{{1,1}}c_{{1,2}}+2\,c_{{
		 		 	 		1,2}}c_{{1,3}}&0&0\end {array} \right) 
		 		 	 ,\\
		 		 	K(e_0,e_2)&=\left( \begin {array}{ccc} 0&-a_{{3,3}}c_{{1,2}}&a_{{1,1}}a_{{3,3}}-c
		 		 	_{{1,3}}a_{{3,3}}\\ \noalign{\medskip}a_{{1,1}}c_{{1,2}}-2\,c_{{1,2}}c
		 		 	_{{1,3}}&0&0\\ \noalign{\medskip}-a_{{1,1}}c_{{1,3}}-{c_{{1,2}}}^{2}+{
		 		 		c_{{1,3}}}^{2}&0&0\end {array} \right),\\
		 		 	K(e_1,e_2)&=\left( \begin {array}{ccc} -2\,a_{{3,3}}c_{{1,2}}&0&0
		 		 	\\ \noalign{\medskip}0&a_{{3,3}}c_{{1,2}}&c_{{1,3}}a_{{3,3}}
		 		 	\\ \noalign{\medskip}0&-c_{{1,3}}a_{{3,3}}&a_{{3,3}}c_{{1,2}}
		 		 	\end {array} \right). 
		 		 		\end{align*}But $K(e_0,e_1)=K(e_0,e_2)=0$ and $K(e_1,e_2)=\frac1{r^2}(E_{32}-E_{23})$. This is equivalent to $c_{1,2}=0$ and $a_{1,1}=c_{1,3}=-\frac1{a_{3,3}r^2}$ and hence $\ga^0$ is given by
		 		 		\[ \ga_{e_0}^0=\left( \begin {array}{ccc} c_{{1,3}}&0&0\\ \noalign{\medskip}0&c_{{1,
		 		 				3}}&0\\ \noalign{\medskip}0&0&c_{{1,3}}\end {array}
		 		 		\right) 
		 		 		,\ga_{e_1}^0=\left( \begin {array}{ccc} 0&-\frac1{c_{1,3}r^2}&0\\ \noalign{\medskip}c_{{1,3}
		 		 		}&0&0\\ \noalign{\medskip}0&0&0\end {array} \right)
		 		 		\esp \ga_{e_2}^0=\left( \begin {array}{ccc} 0&0&-\frac1{c_{1,3}r^2}\\ \noalign{\medskip}0&0&0\\ \noalign{\medskip}c_{{1,3}}&0&0\end {array} \right)
		 		 			.
		 		 			\]One can see that
		 		 			\[ \tr_{\prs_0}(\ga^0)=\left(c_{{1,3}}-\,{\frac {2}{c_{{1,3}}{r}^{2}}}\right)e_0\esp \tr_{\prs_0}((\ga^0)^*)=3c_{{1,3}}e_0. \]
		 			\end{proof}

		 			\begin{pr}\label{prg} Let $(G,\na,\prs)$ be a three dimensional Lie group endowed with a left invariant rigid affine-Riemann structure such that $\prs$ is not flat. Then the Lie algebra of $G$ is isomorphic to $\R^3$ with the non-vanishing Lie brackets $$[e_3,e_1]=e_2,[e_3,e_2]=2e_2$$ and the matrices in $(e_1,e_2,e_3)$ of the metric and the difference tensor at $e$ are given by \[\begin{cases}
		 				\prs=\left( \begin {array}{ccc} 1&\frac12&0\\ \noalign{\medskip}\frac12&1&0
		 				\\ \noalign{\medskip}0&0&\nu\end {array} \right),\;\nu>0\\ 
		 				
		 				 \ga_{e_1}= \left( \begin {array}{ccc} {\frac {\nu\,{r}^{2}+1}{\nu\,r}}&\,{
		 					\frac {2}{\nu\,r}}&0\\ \noalign{\medskip}\frac12\,{\frac {\nu\,{r}^{2}-1}{
		 						\nu\,r}}&{\frac {\nu\,{r}^{2}-1}{\nu\,r}}&0\\ \noalign{\medskip}0&0&r
		 				\end {array} \right),
		 				\ga_{e_2}=\left( \begin {array}{ccc} \,{\frac {2}{\nu\,r}}&\,{\frac {4}{\nu\,
		 						r}}&0\\ \noalign{\medskip}{\frac {\nu\,{r}^{2}-1}{\nu\,r}}&-\,{\frac 
		 					{2}{\nu\,r}}&0\\ \noalign{\medskip}0&0&0\end {array} \right), 
		 				\ga_{e_3}=\left( \begin {array}{ccc} 0&0&\frac4{r}\\ \noalign{\medskip}0&0&-\frac2r
		 				\\ \noalign{\medskip}r&0&0\end {array} \right),\quad r\not=0. 
		 				\end{cases}	\]Moreover,
		 				\[ \tr_{\prs}(\ga)={\frac {2(\nu\,{r}^{2}+6)}{3\nu\,r}}\left( 2e_1-e_2\right)\esp \tr_{\prs}(\ga^*)=2r(2e_1-e_2).
		 				 \]For $\nu=\frac{3}{(3(2^{k-1})-2)r^2}$ we have $\tr_{\prs}(\ga)=(2^k-1)\tr_{\prs}(\ga^*)$ and hence $(T^kG,J_k,g_k)$ is balanced and $(T^{k+1}G,J_{k+1},g_{k+1})$ is Calabi-Yau with torsion.

		 				\end{pr}
		 				
		 				\begin{proof} If $(G,\na,\prs)$ is rigid then, as a Riemannian manifold, it is symmetric and hence it is either irreducible and hence Einstein or  it is the product of $\R$ with a complete Riemannian surface of constant curvature. So the Ricci curvature has signature $(+,+,+)$, $(-,-,-)$,  $(0,+,+)$ or $(0,-,-)$. It is known that no three dimensional Lie group carries a left invariant Riemannian metric of Ricci signature $(0,+,+)$ (see \cite{Milnor}) and if the Ricci signature is 
		 					$(+,+,+)$ then $G$ is compact simple and it is known (see \cite{helms}) that $G$ cannot carry a left invariant flat and torsionless connection. So the Ricci signature is either $(-,-,-)$ of $(0,-,-)$. According to the determination by Lee in \cite[Tables 1 and 2]{lee} of the Ricci  signatures of left invariant metrics on three dimensional Lie groups, the Lie algebra $\G$ of $G$ and the metric are of the following forms:
		 					\begin{enumerate}\item $\G_1=\R^3$,  the non-vanishing Lie brackets are: $[e_3,e_1]=e_1,[e_3,e_2]=e_2$ and $\prs=\mathrm{Diag}(1,1,\nu)$, $\nu>0$,
		 						\item $\G_2=\R^3$, the non-vanishing Lie brackets are: $[e_3,e_1]=e_2,[e_3,e_2]=2e_2$ and $\prs=\left( \begin {array}{ccc} 1&\frac12&0\\ \noalign{\medskip}\frac12&1&0
		 						\\ \noalign{\medskip}0&0&\nu\end {array} \right)$.
		 						
		 						\end{enumerate}
		 					Now, according to Proposition \ref{rigidpr}, $(G,\na,\prs)$ is rigid if and only if the difference tensor at $e$, $\ga:\G\times\G\too\G$ satisfies, for any $u,v\in\G$,
		 					\[ [\mathrm{L}_u,\ga_v]=\ga_{\mathrm{L}_uv}\esp K(u,v)=[\ga_u,\ga_v] \]where $\mathrm{L}_u$ is the left multiplication of the Levi-Civita product given by
		 					\[ 2\langle \mathrm{L}_uv,w\rangle=\langle[u,v],w\rangle+
		 					\langle[w,v],u\rangle+\langle[w,u],v\rangle \]and $K(u,v)=
		 					\mathrm{L}_{[u,v]}-[\mathrm{L}_u,\mathrm{L}_v]$.
		 					
		 					A direct computation using Maple shows that when $\G=\G_1$  there is no solution and in the second case we find the $\ga_{e_i}$  given in the statement of the proposition. The last statement is a consequence of Theorem \ref{infinite}.
		 						\end{proof}

		 	\begin{theo} Let $(M,\na,\prs)$ be simply-connected rigid affine-Riemann manifold with $\prs$ complete. Then:
		 		\begin{enumerate}\item If $\dim M=2$ then $\prs$ is flat.
		 			\item If $\dim M=3$ then either $\prs$ is flat or $(M,\na,\prs)$ isomorphic to $\R\times S^2(r)$ endowed with the rigid structure given in Proposition \ref{le} or to the left invariant rigid structure given in Proposition \ref{prg}.
		 			
		 			\end{enumerate}

		 		\end{theo}
		 		
		 	\begin{proof} Suppose that $\dim M=2$. Note first that the vector fields $(\tr_{\prs}(\ga),\tr_{\prs}(\ga^*))$ are $D$-parallel. If $\tr_{\prs}(\ga)=\tr_{\prs}(\ga^*)=0$ then we can apply Theorem \ref{trga=02} to get that the curvature of $\prs$ vanishes. If $(\tr_{\prs}(\ga),\tr_{\prs}(\ga^*))\not=(0,0)$ then there is a non zero $D$-parallel vector field on $M$ and hence the curvature of $D$ vanishes.

		 	Suppose now that $\dim M=3$ and $(M,\prs)$ is not flat. Since $M$ is compact and carries an affine structure it is not compact. Then $(M,\prs)$ is a non compact simply-connected symmetric Riemannian manifold and hence it is the Riemannian product of a Euclidean space and a finite number of irreducible symmetric spaces (see \cite[ Theorem 7.76 pp. 194]{besse}). Then $(M,\prs)$ is either irreducible or it is the product of $\R$ with a complete Riemannian surface of constant curvature. If $(M,\prs)$ is irreducible then it is Einstein with nonpositive scalar curvature $s$. If $s=0$ then $(M,\prs)$ it is Ricci-flat and hence flat since any homogeneous Ricci-flat Riemannian manifold is flat (see \cite{alek}). If $s<0$ then, according Alekseevskii conjecture which is true in dimension $\leq5$ (see \cite[Conjecture 7.57 pp. 190]{besse}), $(M,\prs)$ is isometric to a solvable Lie group with a left invariant metric.  If $(M,\prs)$ is the product of $\R$ with a complete Riemannian surface $S$ of constant curvature then $S$ is either the 2-dimensional hyperbolic pace $H^2$ or $S^2(r)$ endowed with their canonical metric.  When  $S=H^2$ then $(M,\prs)$ is isometric to a solvable Lie group with a left invariant metric. So far, we have shown that if $(M,\prs)$ is not flat then 	$(M,\prs)$ is isometric to a 3-dimensional solvable Lie group with a left invariant Riemannian metric or $\R\times S^2(r)$. 
		 	
		 	Suppose that $(M,\prs)$ is isometric to a 3-dimensional solvable Lie group with a left invariant metric. Let us show that $\na$ is also left invariant. According to Theorem \ref{rigid}, there exists a simply-connected Lie group $G$ which act transitively on $M$ and preserves both $\na$ and $\prs$. From the determination of the isometry groups of 3-dimensional solvable Lie groups (see \cite{lee, cosgaya}) one can see that the dimension of the isometry group of $(M,\prs)$ is either 3 or 4 and hence $\dim G=3$ or $4$. If $\dim G=4$ then $G$ contains the left multiplications and hence $\na$ is left invariant. If $\dim G=3$, the orbital map $ev:G\too M$, $h\mapsto h(e)$ is a covering and hence a diffeomorphism since both $G$ and $M$ are simply-connected. Moreover, $ev$ commutes with the actions of $G$ by left multiplication on $G$ and its natural action on $M$. If we pull-back the metric $\prs$ and $\na$ on $G$, we get that $(M,\na,\prs)$ is isomorphic to a Lie group with a left invariant connection and a left invariant metric. To complete the proof, we apply Propositions \ref{le} and \ref{prg}.
		 				\end{proof}	
		 
		 \section{ Infinitely balanced affine-Riemann manifolds }\label{section6}
		 
		 In this section, we introduce the notion of {\it infinitely balanced} affine-Riemann manifold (see Definition \ref{def}). we illustrate the importance of this class of affine-Riemann manifolds and we give some of its properties.

		 \begin{Def}\label{def} We call an affine-Riemann manifold $(M,\na,\prs)$  infinitely
		 	balanced if its difference tensor $\ga$ satisfies $\tr_{\prs}(\ga)=
		 	\tr_{\prs}(\ga^*)=0$. This is equivalent to the Koszul forms satisfying $\al=\xi=0$.

		 \end{Def}
		 
		 This definition find its justification in the following result which is a consequence of Theorem \ref{infinite} item 3 and Propositions 
		 \ref{pryau}-\ref{prchern}.
		 \begin{theo}\label{infbal} Let $(M,\na,\prs)$ be an infinitely balanced affine-Riemann manifold with $\ga\not=0$. Then:
		 	\begin{enumerate}
		 		\item $(TM,J_1,g_1)$ is balanced with $\rho^B=\rho^C=0$ and it is K\"ahler if and only if $\ga=\ga^*$.
		 		\item For any $k\geq2$, $(T^kM,J_k,g_k)$ is balanced non-K\"ahler with $\rho^B=\rho^C=0$.
		 	\end{enumerate}
		 \end{theo}
		 
		 \begin{exem} \item In Table \ref{table5}, we give many examples of  infinitely balanced left invariant structures on some 3-dimensional Lie groups.
		 	
		 	\end{exem}
		 
		 Let us start by the following remark.
		 \begin{pr} Let $(M,\na,\prs)$ a compact affine-Riemann manifold such that $(TM,J_1,g_1)$ is Gauduchon and $\al=0$. Then $\xi=0$ and hence $(M,\na,\prs)$ is infinitely balanced.

		 \end{pr}
		 \begin{proof} It is an immediate consequence of \eqref{gaud} and the fact that $\int_Md^*(\al-\xi)=0$.
		 \end{proof}
		 
		 The following result describes completely the situation in dimension 2.
		 
		 \begin{theo}\label{trga=02} Let $(M,\na,\prs)$ be a connected 2-dimensional infinitely balanced affine-Riemann manifold. Then $\prs$ is Hessian, i.e., $\ga=\ga^*$ and its sectional curvature is nonnegative. 
		 	Moreover, if $\prs$ is complete then $\na$ is the Levi-Civita connection of $\prs$ and $M$ is either  diffeomorphic to   the torus $\mathbb{T}^2$ or $\R^2$. 
		 				 	
		 \end{theo}
		 
		 \begin{proof}Let us first show that $\prs$ is Hessian. Choose an orthonormal frame $(E_1,E_2)$. We have in the basis $(E_1,E_2)$, since $\ga_{E_1}E_2=\ga_{E_2}E_1$,
		 	\[ \ga_{E_1}=\left(\begin{array}{cc}\ga_{11}^1&\ga_{12}^1\\\ga_{11}^2&\ga_{12}^2 \end{array}     \right)\esp \ga_{E_2}=\left(\begin{array}{cc}\ga_{12}^1&\ga_{22}^1\\\ga_{12}^2&\ga_{22}^2 \end{array}     \right).  \]
		 	The condition $\tr_{\prs}(\ga^*)=\tr_{\prs}\ga=0$ is equivalent to
		 	\[ \ga_{11}^1+\ga_{12}^2=\ga_{12}^1+\ga_{22}^2=\ga_{11}^1+\ga_{22}^1=\ga_{11}^2+\ga_{22}^2=0. \]Thus
		 	\[ \ga_{E_1}=\left(\begin{array}{cc}\ga_{11}^1&\ga_{11}^2\\\ga_{11}^2&-\ga_{11}^1 \end{array}     \right)\esp \ga_{E_2}=\left(\begin{array}{cc}\ga_{11}^2&-\ga_{11}^1\\-\ga_{11}^1&-\ga_{11}^2 \end{array}     \right).  \]
		 	This shows that $\ga=\ga^*$ and hence $\prs$ is Hessian. According to to \eqref{cu1} and \eqref{cu}, the curvature of $\prs$ is given by
		 	\[ K(X,Y)=[\ga_X,\ga_Y]. \]	
		 	So the curvature $\kappa:M\too\R$ of $\prs$ is given by
		 	\[ \kappa=\langle K(E_1,E_2)E_1,E_2\rangle=\tr(\ga_{E_1}^2)\geq0. \]If $M$ is compact,	according to Gauss-Bonnet's theorem,
		 	\[ \int_M\kappa\nu=2\pi(2-2g)\geq0 \]and hence $g\leq1$. But the case $g=0$ is not possible since the 2-sphere has no affine structure and hence $g=1$,  $\kappa=0$ and then $\ga=0$. If $M$ is non compact and $\prs$ is complete then according to a theorem of Cohen-Vossen \cite{cohn2} $M$ is diffeomorphic to $\R^2$. But a theorem of Cheng-Yau and Pogorelov (see \cite[Theorem 8.6 pp. 160]{Shima}) asserts that the only Hessian metric on $\R^n$ which satisfies $\tr_{\prs}(\ga)=0$ is the canonical metric.
		 	 \end{proof}

		 In dimension superior to 3, we have the following theorem about infinitely balanced Hessian affine-Riemann manifolds.
		 \begin{theo}\label{trga=0n} Let $(M,\na,\prs)$ be an  affine-Riemann manifold satisfying $\ga=\ga^*$ and  $\tr_{\prs}(\ga)=0$. Then the Ricci curvature of $\prs$ is nonnegative and $\prs$ is Ricci-flat if and only if $\ga=0$. Moreover, if $M$ is compact  then
		 	  $\ga=0$ and hence $\na$ is the Levi-Civita connection of $\prs$.
		 	
		 \end{theo}
		 
		 \begin{proof} The condition $\ga=\ga^*$ implies by virtue of \eqref{cu1} and \eqref{cu} that, for any $X,Y\in\Ga(TM)$,
		 	\[ K(X,Y)=[\ga_X,\ga_Y]. \]Since $\tr_{\prs}(\ga)=0$, we get that the Ricci curvature is given by
		 	\[ \ric(X,X)=\tr(\ga_X^2)\geq0 \]and $\ric(X,X)=0$ if and only if $\ga_X=0$.
		 	If $M$ is compact then according to \cite[Theorem 8.8 pp. 162]{Shima} $\ga=0$ and $\na$ is the Levi-Civita connection of $\prs$.
		 		\end{proof}
		 If we drop the hypothesis $M$ compact there are non trivial  infinitely balanced Hessian affine-Riemannian manifolds.
		 
		 \begin{theo}\label{exemple} Let $n\geq2$ and $c>0$. On $\R^n\setminus\{0\}$ endowed with its affine connection $\na^0$ and its canonical Euclidean product $\prs_0$, consider the smooth function $$\di{}f(x_1,\ldots,x_n)=
		 	\int_0^{r}(t^n+1)^{\frac1n}dt$$ where $r=\sqrt{x_1^2+\ldots+x_n^2}$ and the matrix 
		 	$\prs=\left(\frac{\partial^2f}{\partial_{x_i}\partial_{x_j}} \right)_{1\leq i,j\leq n}$. Then $\prs$ is a Riemannian metric and $(\R^{n}\setminus\{0\},\na^0,\prs)$ is an affine-Riemann manifold satisfying $\ga=\ga^*$ and $\tr_{\prs}(\ga)=0$. Moreover, for any $i\not=j$ and any $u,v\in T_X(\R^n\setminus\{0\})$
		 	\[\begin{cases} \langle\partial_{x_i},\partial_{x_i}\rangle=\frac{r^{{n+2}}+c(r^2-x_i^2)}{r^{3}(r^{{n}}+c)^{\frac{n-1}n}},\; 
		 	\langle\partial_{x_i},\partial_{x_j}\rangle=
		 	-\frac{cx_ix_j}{r^{3}(r^{{n}}+c)^{\frac{n-1}n}},\\ 
		 	\langle u,v\rangle=\frac{1}{r^{3}(r^{{n}}+c)^{\frac{n-1}n}}\left(
		 	(r^n+c)r^2\langle u,v\rangle_0-c\langle u,X\rangle_0\langle v,X\rangle_0  \right),\;X=(x_1,\ldots,x_n) \end{cases}\]and the Ricci curvature of $\prs$ is nonnegative.
		 	
		 	\end{theo}
		 	
		 	\begin{proof}The function $f$ is smooth on $\R^{n}\setminus\{0\}$ and it is easy to show that, for any $i\not=j$,
		 	\[ \frac{\partial^2f}{\partial {x_i^2}}=\frac{r^{{n+2}}+c(r^2-x_i^2)}{r^{3}(r^{{n}}+c)^{\frac{n-1}n}}\esp 
		 	\frac{\partial^2f}{\partial_{x_i}\partial_{x_j}}=
		 	-\frac{cx_ix_j}{r^{3}(r^{{n}}+c)^{\frac{n-1}n}}. \]	
		 		 Let $(e_1,\ldots,e_n)$ be the canonical basis of $\R^n$. Let us show that $\prs$ is definite positive. For any $X=(x_1,\ldots,x_n)\in\R^{n}\setminus\{0\}$ $u\in T_X(\R^{n}\setminus\{0\})$,
		 		We have
		 		\begin{align*}
		 		({r^{3}(r^{{n}}+c)^{\frac{n-1}n}})\langle u,u\rangle&=\sum_{i=1}^n(r^2(r^n+c)-cx_i^2)u_i^2
		 		-c\sum_{i\not=j}x_ix_ju_iu_j\\
		 		&=(r^2(r^n+c)\sum_{i=1}^nu_i^2-c(x_1u_1+\ldots+x_nu_n)^2\\
		 		&=(r^n+c)|X|_0^2|u|_0^2-c\langle u,X\rangle_0^2.		 
		 		\end{align*}By virtue of Schwartz inequality
		 		\[ c\langle u,X\rangle_0^2\leq c|X|_0^2|u|_0^2\leq (r^n+c)|X|_0^2|u|_0^2.  \]This shows that $\langle u,u\rangle\geq0$ and $\langle u,u\rangle=0$ if and only if $u=0$. Then $\prs$ is a Hessian metric. 
		 		
		 		 Let us show that $\det(\prs)=1$. Indeed, the rows $(L_1,\ldots,L_n)$ of $\prs$ are given by
		 		\[ L_i=\frac1{{r^{3}(r^{{n}}+c)^{\frac{n-1}n}}}(r^2(r^n+c)e_i-cx_i\rho)\esp \rho=x_1e_1+\ldots+x_ne_n. \]So
		 		\begin{align*}
		 		(r^{3n}(r^{{n}}+c)^{{n-1}})\det(\prs)&=\det(r^2(r^n+c)e_1-cx_1\rho,\ldots,r^2(r^n+c)e_n-cx_n\rho)\\
		 		&=r^{2n}(r^n+c)^n+\sum_{i=1}^n\det(r^2(r^n+c)e_1,\ldots,
		 		r^2(r^n+c)e_{i-1},-cx_i\rho,r^2(r^n+c)e_{i+1},\ldots,r^2(r^n+c)e_{n})\\
		 		&=r^{2n}(r^n+c)^n-cr^{2n}(r^n+c)^{n-1}\\
		 		&=r^{2n}(r^n+c)^{n-1}(r^n+c-c)=r^{3n}(r^n+c)^{n-1}.
		 		\end{align*}So $\tr_{\prs}(\ga)=0$ which completes the proof.
		 		\end{proof}

		\section{Some examples of  left invariant generalized K\"ahler structures on some 6-dimensional connected and simply-connected Lie groups}\label{section7}

		In this section, we give examples of left invariant generalized K\"ahler structures on some 6-dimensional connected and simply-connected Lie groups by giving the complex isomorphism and the metric on the corresponding Lie algebras.  
		
		Our examples are based on the classification of 3-dimensional real Novikov algebras given in \cite{burde}. Recall that a Novikov algebra is a left symmetric algebra such that right multiplications commute.
		
		Let $(\g,.,\prs_0)$ be a Novikov algebra of dimension 3 endowed with a scalar product. The bracket $[a,b]=a.b-b.a$ induces on $\g$ a Lie algebra structure. Let $G$ be the connected and simply-connected Lie group associated to $(\G,\br)$. Then the left symmetric product and $\prs_0$ induce on $G$ a left invariant affine-Riemann structure $(\na,\prs)$. We have seen that on $\Phi(\G)=\g\times\G$ there are a Lie  bracket $\br_\Phi$, a complex isomorphism $J$ and a scalar product $\prs_\Phi$ given by \eqref{tg} and \eqref{J}. Moreover, according to Theorem \ref{main}, the Hermitian structure $(TG,J_1,g_1)$ associated to $(G,\na,\prs)$ is diffeomorphic to $(G\times\G,J_0,g_0)$ where $G\times\G$ is the simply-connected Lie group associated to $(\Phi(\G),\br_\Phi)$ and $(J_0,g_0)$ are the left invariant tensor field associated to $(J,\prs_\Phi)$.
		
		In Tables \ref{table1} and \ref{table2}, for any 3-dimensional real Novikov algebra given in \cite{burde} and identified to $\R^3$ with its canonical basis $(e_1,e_2,e_3)$, we give its multiplication table and the Lie bracket $\br_\Phi$ on $\Phi(\R^3)=\R^3\times\R^3$ in the basis $(f_1,\ldots,f_6)$ where $f_i=(e_i,0)$ for $1=1,2,3$ and $f_j=(0,e_j)$ for $j=4,5,6$. These 6-dimensional Lie algebras are labeled $N_1^{\G_1}(a)$, $N_2^{\G_1}$ and so on. 
		The metric $\prs_0$ is given by its matrix in $(e_1,e_2,e_3)$,
		\[ \prs_0=\left(\begin{array}{ccc}g_{1,1}&g_{1,2}&g_{1,3}\\
		    g_{1,2}&g_{2,2}&g_{2,3}\\
		    g_{1,3}&g_{2,3}&g_{3,3}\end{array}    \right). \]
		In Tables \ref{table3}-\ref{table8}, when we refer to a 6-dimensional Lie algebra in Tables \ref{table1} and \ref{table2} having a generalized K\"ahler structure this means that $(J,\prs_\Phi)$ are given in the basis $(f_1,\ldots,f_6)$ by
		\[ J=\left( \begin{array}{cc}0&-\mathrm{I}_3\\\mathrm{I}_3&0\end{array}     \right)\esp \prs_\Phi= \left( \begin{array}{cc}\prs_0&0\\0&\prs_0\end{array}     \right) \]with the mentioned  restrictions on the $(g_{i,j})$.

		The realization of the examples in Tables \ref{table3}-\ref{table8} was possible thank to the software Maple.
		
			{\renewcommand*{\arraystretch}{1.1}
				\begin{center}
					
					\begin{tabular}{|l|l|}
						\hline
						Left symmetric product on $\R^3$&$e_1\bullet e_1=ae_1,e_1\bullet e_2=(1+a)e_2,e_1\bullet e_3=(1+a)e_3,e_2\bullet e_1=ae_2,e_3\bullet e_1=ae_3.$\\
						\hline
						\multirow{2}{*}{	Lie brackets on $N_1^{\G_1}(a)$}&$[f_1,f_2]=f_2,[f_1,f_3]=f_3,[f_1,f_4]=af_4,[f_1,f_5]=(1+a)f_5,
						[f_1,f_6]=(1+a)f_6,$\\
						&$[f_2,f_4]=af_5,[f_3,f_4]=af_6.$\\
						\hline
						\hline
						Left symmetric product on $\R^3$&$e_1\bullet e_1=-e_1+e_2,
						e_2\bullet e_1=-e_2,e_3\bullet e_1=-e_3.$\\
						\hline
						\multirow{1}{*}{	Lie brackets on $N_2^{\G_1}$}&$[f_1,f_2]=f_2,[f_1,f_3]=f_3,[f_1,f_4]=-f_4+f_5,
						,[f_2,f_4]=-f_5,[f_3,f_4]=-f_6.$\\
						\hline
						\hline
						Left symmetric product on $\R^3$&$e_1\bullet e_1=ae_1,e_1\bullet e_2=ae_2+e_3,e_1\bullet e_3=\al e_2+(1+a)e_3,e_2\bullet e_1=ae_2,e_3\bullet e_1=ae_3.$\\
						\hline
						\multirow{2}{*}{	Lie brackets on $N_2^{\G_2^\al}(a)$}&$[f_1,f_2]=f_3,[f_1,f_3]=\al f_2+f_3,[f_1,f_4]=af_4,[f_1,f_5]=a f_5+f_6,$\\
						&$[f_1,f_6]=\al f_5+(1+a)f_6,$
						$[f_2,f_4]=af_5,[f_3,f_4]=af_6.$\\
						\hline
						\hline
						\multirow{2}{*}	{Left symmetric product on $\R^3$}&$e_1\bullet e_1=ae_1+e_2,e_1\bullet e_2=ae_2+e_3,e_1\bullet e_3=(a+a^2) e_2+(1+a)e_3,$\\&$e_2\bullet e_1=ae_2,e_3\bullet e_1=ae_3.$\\
						\hline
						\multirow{2}{*}{	Lie brackets on $N_2^{\G_2^{a^2+a}}(a)$}&$[f_1,f_2]=f_3,[f_1,f_3]=(a+a^2) f_2+f_3,[f_1,f_4]=af_4+f_5,[f_1,f_5]=a f_5+f_6,$\\
						&$[f_1,f_6]=(a+a^2) f_5+(1+a)f_6,$
						$[f_2,f_4]=af_5,[f_3,f_4]=af_6.$\\
						\hline
						\hline
						\multirow{2}{*}	{Left symmetric product on $\R^3$}&$e_1\bullet e_1=-\frac13e_1,e_1\bullet e_2=\frac83e_2-8e_3,
						e_1\bullet e_3=\frac79e_2-\frac73e_3,$\\&$e_2\bullet e_1=\frac83e_2-9e_3,e_3\bullet e_1=e_2-\frac{10}3e_3.$\\
						\hline
						\multirow{2}{*}{	Lie brackets on $N_3^{\G_2^{-\frac29}}$}&
						$[f_1,f_2]=f_3,
						[f_1,f_3]=-\frac29f_2+f_3,[f_1,f_4]=-\frac13f_4,[f_1,f_5]=\frac83f_5-8f_6,$\\
						&$[f_1,f_6]=\frac79f_5-\frac73f_6,$
						$[f_2,f_4]=\frac83f_5-9f_6,[f_3,f_4]=f_5-\frac{10}3f_6.$\\
						\hline
						\hline
						\multirow{2}{*}	{Left symmetric product on $\R^3$}&$e_1\bullet e_1=-\frac13e_1+e_2,e_1\bullet e_2=\frac83e_2-8e_3,
						e_1\bullet e_3=\frac79e_2-\frac73e_3,$\\&$e_2\bullet e_1=\frac83e_2-9e_3,e_3\bullet e_1=e_2-\frac{10}3e_3.$\\
						\hline
						\multirow{2}{*}{	Lie brackets on $N_4^{\G_2^{-\frac29}}$}&
						$[f_1,f_2]=f_3,
						[f_1,f_3]=-\frac29f_2+f_3,[f_1,f_4]=-\frac13f_4+e_5,[f_1,f_5]=\frac83f_5-8f_6,$\\
						&$[f_1,f_6]=\frac79f_5-\frac73f_6,$
						$[f_2,f_4]=\frac83f_5-9f_6,[f_3,f_4]=f_5-\frac{10}3f_6.$\\
						\hline
						\hline
						\multirow{3}{*}	{Left symmetric product on $\R^3$}&
						$e_1\bullet e_1=3ae_1-(3a^2+\frac13a)e_2,e_1\bullet e_2=6ae_2+(1-9a)e_3,e_1\bullet e_3=(a-\frac29)e_2+e_3,$\\&
						$e_2\bullet e_1=6ae_2-9ae_3,e_2\bullet e_2=-3e_2+9e_3,e_2\bullet e_3=-e_2+3e_3,$\\&
						$e_3\bullet e_1=ae_2,e_3\bullet e_2=-e_2+3e_3,e_3\bullet e_3=-\frac13e_2+e_3$.\\
						\hline
						\multirow{3}{*}{	Lie brackets on $N_5^{\G_2^{-\frac29}}(a)$}&
						$[f_1,f_2]=f_3,
						[f_1,f_4]=3af_4-(3a^2+\frac13a)f_5,[f_1,f_5]=6af_5+(1-9a)f_6,$\\
						&$[f_1,f_3]=-\frac29f_2+f_3,
						[f_1,f_6]=(a-\frac29)f_5+f_6,
						[f_2,f_4]=6af_5-9af_6,[f_2,f_5]=-3f_5+f_6,$
						\\&$[f_2,f_6]=-f_5+3f_6,[f_3,f_4]=af_5,[f_3,f_5]=-f_5+3f_6,[f_3,f_6]=-\frac13f_5+f_6.$\\
						\hline
						\hline
						\multirow{3}{*}	{Left symmetric product on $\R^3$}&
						$e_1\bullet e_1=-\frac23e_1-\frac8{27}e_2+\frac23e_3,e_1\bullet e_2=-\frac43e_2+3e_3,e_1\bullet e_3=-\frac49e_2+e_3,$\\
						&$e_2\bullet e_1=-\frac43e_2+2e_3,e_2\bullet e_2=-3e_2+9e_3,e_2\bullet e_3=-e_2+3e_3,$\\
						&$e_3\bullet e_1=-\frac29e_2,e_3\bullet e_2=-e_2+3e_3,
						e_3\bullet e_3=-\frac13e_2+e_3,$\\
						\hline
						\multirow{3}{*}{	Lie brackets on $N_6^{\G_2^{-\frac29}}$}&
						$[f_1,f_2]=f_3,
						[f_1,f_3]=-\frac29f_2+f_3,[f_1,f_4]=-\frac23f_4-\frac8{27}f_5+\frac23f_6,
						[f_1,f_5]=-\frac43f_5+3f_6,$\\
						&$[f_1,f_6]=-\frac49f_5+f_6,
						[f_2,f_4]=-\frac43f_5+2f_6,[f_2,f_5]=-3f_5+9f_6,[f_2,f_6]=-f_5+3f_6,$\\
						&$[f_3,f_4]=-\frac29f_5,[f_3,f_5]=-f_5+3f_6,[f_3,f_6]=-\frac13f_5+f_6,$\\
						\hline
						\hline
						\multirow{3}{*}	{Left symmetric product on $\R^3$}&
						$e_1\bullet e_1=-\frac23e_1-\frac{11}{27}e_2+e_3,e_1\bullet e_2=-\frac43e_2+3e_3,e_1\bullet e_3=-\frac49e_2+e_3,$\\
						&$e_2\bullet e_1=-\frac43e_2+2e_3,e_2\bullet e_2=-3e_2+9e_3,e_2\bullet e_3=-e_2+3e_3,$\\
						&$e_3\bullet e_1=-\frac29e_2,e_3\bullet e_2=-e_2+3e_3,
						e_3\bullet e_3=-\frac13e_2+e_3,$\\
						\hline
						\multirow{3}{*}{	Lie brackets on $N_7^{\G_2^{-\frac29}}$}&
						$[f_1,f_2]=f_3,
						[f_1,f_3]=-\frac29f_2+f_3,[f_1,f_4]=-\frac23f_4-\frac{11}{27}f_5+f_6,
						[f_1,f_5]=-\frac43f_5+3f_6,$\\
						&$[f_1,f_6]=-\frac49f_5+f_6,
						[f_2,f_4]=-\frac43f_5+2f_6,[f_2,f_5]=-3f_5+9f_6,[f_2,f_6]=-f_5+3f_6,$\\
						&$[f_3,f_4]=-\frac29f_5,[f_3,f_5]=-f_5+3f_6,[f_3,f_6]=-\frac13f_5+f_6,$\\
						\hline					
					\end{tabular}\captionof{table}{Three dimensional Novikov algebras and their associated phase Lie algebras.\label{table1}} \end{center}}

			{\renewcommand*{\arraystretch}{1.2}
				\begin{center}
					
					\begin{tabular}{|l|l|}

						\hline
						\multirow{1}{*}	{Left symmetric product on $\R^3$}
						&$e_1\bullet e_1=ae_1,e_1\bullet e_2=(1+a)e_3,e_1\bullet e_3=(1+a)e_3,e_2\bullet e_1=ae_3,
						e_3\bullet e_1=ae_3.$\\
						\hline
						\multirow{2}{*}{	Lie brackets on $N_8^{\G_2^{0}}(a)$}&
						$[f_1,f_2]=f_3,
						[f_1,f_3]=f_3,[f_1,f_4]=af_4,[f_1,f_5]=(1+a)f_6,
						[f_1,f_6]=(1+a)f_6,$\\&$
						[f_2,f_4]=af_6,[f_3,f_4]=af_6.$\\
						\hline
						\hline
						\multirow{1}{*}	{Left symmetric product on $\R^3$}&$e_1\bullet e_1=-e_1+e_3,e_2\bullet e_1=-e_3,e_3\bullet e_1=-e_3,$\\
						\hline
						\multirow{1}{*}{	Lie brackets on $N_9^{\G_2^{0}}$}&$[f_1,f_2]=f_3,
						[f_1,f_3]=f_3,[f_1,f_4]=-f_4+f_6,[f_2,f_4]=-f_6,
						[f_3,f_4]=-f_6,$
						\\
						\hline
						\hline
						\multirow{1}{*}	{Left symmetric product on $\R^3$}&
						$e_1\bullet e_1=ae_1,e_1\bullet e_2=ae_2+e_3,e_1\bullet e_3=(1+a)e_3,$\\&$e_2\bullet e_1=ae_2,e_2\bullet e_2=-e_2+e_3,
						e_3\bullet e_1=ae_3$\\
						\hline
						\multirow{2}{*}{	Lie brackets on $N_{10}^{\G_2^{0}}(a)$}
						&$[f_1,f_2]=f_3,
						[f_1,f_3]=f_3,[f_1,f_4]=af_4,[f_1,f_5]=af_5+f_6,$\\
						&$[f_1,f_6]=(1+a)f_6,$
						$[f_2,f_4]=af_5,[f_2,f_5]=-f_5+f_6,[f_3,f_4]=af_6.$\\
						\hline
						\hline
						\multirow{1}{*}	{Left symmetric product on $\R^3$}&
						$e_1\bullet e_1=-e_1+e_3,e_1\bullet e_2=-e_2+e_3,
						e_2\bullet e_1=-e_2,e_2\bullet e_2=-e_2+e_3,
						e_3\bullet e_1=-e_3$\\
						\hline
						\multirow{2}{*}{	Lie brackets on $N_{11}^{\G_2^{0}}$}&
						$[f_1,f_2]=f_3,
						[f_1,f_3]=f_3,[f_1,f_4]=-f_4+f_5,[f_1,f_5]=-f_5+f_6,$\\
						&
						$[f_2,f_4]=-f_5,[f_2,f_5]=-f_5+f_6,[f_3,f_4]=-f_6.$\\
						\hline
						\hline
						\multirow{1}{*}	{Left symmetric product on $\R^3$}&
						$e_1\bullet e_1=e_2,e_1\bullet e_2=(1+a)e_3,e_2\bullet e_1=ae_3,$\\
						\hline
						\multirow{1}{*}{	Lie brackets on $N_1^{\G_3}(a)$}&
						$[f_1,f_2]=f_3,
						[f_1,f_4]=f_5,[f_1,f_5]=(1+a)f_6,[f_2,f_4]=af_6$\\
						\hline
						\hline
						\multirow{1}{*}	{Left symmetric product on $\R^3$}&
						$e_1\bullet e_1=ae_3,e_1\bullet e_2=e_3,e_2\bullet e_2=e_3,$\\
						\hline
						\multirow{1}{*}{	Lie brackets on $N_2^{\G_3}(a)$}&
						$[f_1,f_2]=f_3,
						[f_1,f_4]=af_6,[f_1,f_5]=f_6,[f_2,f_5]=f_6$\\
						\hline
						\hline
						\multirow{1}{*}	{Left symmetric product on $\R^3$}&
						$e_1\bullet e_1=e_3,e_1\bullet e_2=e_1,e_2\bullet e_1=e_1-e_3,e_2\bullet e_2=e_2,e_2\bullet e_3=e_3,e_3\bullet e_2=e_3$\\
						\hline
						\multirow{2}{*}{	Lie brackets on $N_3^{\G_3}$}&
						$[f_1,f_2]=f_3,
						[f_1,f_4]=f_6,[f_1,f_5]=f_4,
						[f_2,f_4]=f_4-f_6,[f_2,f_5]=f_5,$\\&$[f_2,f_6]=f_6,[f_3,f_5]=f_6.$\\
						\hline	
						\hline
						\multirow{1}{*}	{Left symmetric product on $\R^3$}&
						$e_1\bullet e_2=e_1,e_2\bullet e_1=e_1-e_3,e_2\bullet e_2=e_2,e_2\bullet e_3=e_3,e_3\bullet e_2=e_3$\\
						\hline
						\multirow{2}{*}{	Lie brackets on $N_4^{\G_3}$}&
						$[f_1,f_2]=f_3,
						[f_1,f_5]=f_4,
						[f_2,f_4]=f_4-f_6,[f_2,f_5]=f_5,$\\&$[f_2,f_6]=f_6,[f_3,f_5]=f_6.$\\
						\hline
						\hline
						\multirow{1}{*}	{Left symmetric product on $\R^3$}&
						$e_1\bullet e_2=\frac12e_3,e_2\bullet e_1=-\frac12e_3.$\\
						\hline
						\multirow{1}{*}{	Lie brackets on $N_5^{\G_3}$}&
						$[f_1,f_2]=f_3,
						[f_1,f_5]=\frac12f_6,
						[f_2,f_4]=-\frac12f_6.$\\
						\hline
						\hline
						\multirow{1}{*}	{Left symmetric product on $\R^3$}&
						$e_1\bullet e_1=ae_1,e_1\bullet e_2=ae_2+e_3,e_1\bullet e_3=e_2+ae_3,e_2\bullet e_1=ae_2,e_3\bullet e_1=ae_3$\\
						\hline
						\multirow{2}{*}{	Lie brackets on $N_1^{\G_4}(a)$}&$[f_1,f_2]=f_3,
						[f_1,f_3]=f_2,[f_1,f_4]=af_4,[f_1,f_5]=af_5+f_6,$\\
						&$[f_1,f_6]=f_5+af_6,$
						$[f_2,f_4]=af_5,[f_3,f_4]=af_6.$\\
						\hline
						\hline
						\multirow{1}{*}	{Left symmetric product on $\R^3$}&
						$e_1\bullet e_1=e_1+e_3,e_1\bullet e_2=e_2+e_3,e_1\bullet e_3=e_2+e_3,e_2\bullet e_1=e_2,e_3\bullet e_1=e_3$\\
						\hline
						\multirow{2}{*}{	Lie brackets on $N_2^{\G_4}$}&$[f_1,f_2]=f_3,
						[f_1,f_3]=f_2,[f_1,f_4]=f_4+f_6,[f_1,f_5]=f_5+f_6,$\\
						&$[f_1,f_6]=f_5+f_6,$
						$[f_2,f_4]=f_5,[f_3,f_4]=f_6.$\\
						\hline	
						\hline
						\multirow{1}{*}	{Left symmetric product on $\R^3$}&
						$e_1\bullet e_1=ae_1,e_1\bullet e_2=ae_2+e_3,e_1\bullet e_3=-e_2+ae_3,e_2\bullet e_1=ae_2,e_3\bullet e_1=ae_3$\\
						\hline
						\multirow{2}{*}{	Lie brackets on $N_1^{\G_5}(a)$}&$[f_1,f_2]=f_3,
						[f_1,f_3]=-f_2,[f_1,f_4]=af_4,[f_1,f_5]=af_5+f_6,$\\
						&$[f_1,f_6]=-f_5+af_6,$
						$[f_2,f_4]=af_5,[f_3,f_4]=af_6.$\\
						\hline
						
					\end{tabular}\captionof{table}{Three dimensional Novikov algebras and their associated phase Lie algebras (Continued).\label{table2}}\end{center}}					
			
		{\renewcommand*{\arraystretch}{1.2}
			\begin{center}
				
				\begin{tabular}{|l|l|}
					\hline
					The Lie algebra& Conditions on the Hermitian metric	\\
					\hline
					$N_1^{\G_1}(a)$&$\left[a=-2,g_{1,2}=g_{1,3}=0, 0<g_{{1,1}},0<g_{{3,3}},{\frac {{g_{{2,3}}}^{2}}{
							g_{{3,3}}}}<g_{{2,2}} \right]
					$\\
					\hline
					\multirow{2}{*}{$N_2^{\G_2^\al}(a)$}&
					$\left[a=-2,\al=0,g_{1,3}=0,g_{2,3}=g_{2,2},0<g_{{2,2}},g_{{2,2}}<g_{{3,3}},-{\frac {{g_{{1,2
									}}}^{2}g_{{3,3}}}{g_{{2,2}} \left( -g_{{3,3}}+g_{{2,2}} \right) }}<g_{
								{1,1}}  
							   \right]$ \\ &$\left[a=-1,\al<0,g_{1,2}=g_{1,3}=0,g_{2,2}=-\frac{g_{3,3}}{\al},g_{2,3}=\frac{g_{2,2}}2,0<g_{{1,1}},0<g_{{2,2}},\alpha<-\frac14   \right]$\\
					\hline		
					$N_8^{\G_2^{0}}(a)$&$\left[ a=-2,g_{{1,2}}=0,g_{{1,3}}=0,g_{{2,3}}=g_{{3,3}}, 0<g_{{1,1}},0<g_{{2,3}},g_{{2,3}}<g_{{2,2}} \right]$\\
					\hline
					$N_{10}^{\G_2^{0}}(a)$&$\left[ a=-2,g_{{1,2}}=2\,g_{{2,2}}-2\,g_{{2,3}},g_{{1,3}}=0,g_{{3,3}
					}=g_{{2,3}}, 0<g_{{2,3}},g_{{2,3}}<g_{{2,2}},4\,g_{{2,2}}-4\,g_{{2,3}}<g_{
					{1,1}}
				 \right]
					$\\
					\hline
					$N_1^{\G_5}(a)$	&$\left[ a=0,g_{{1,2}}=0,g_{{1,3}}=0,g_{{2,2}}=g_{{3,3}},g_{{2,3}}=0,0<g_{1,1},0<g_{2,2}
					\right] 
					$\\
					\hline			
				\end{tabular}	\captionof{table}{Examples of K\"ahler Lie algebras.\label{table3}}\end{center}}	
		
			{\renewcommand*{\arraystretch}{1.4}
				\begin{center}
					
					\begin{tabular}{|l|l|}
						\hline
						The Lie algebra& Conditions on the Hermitian metric	\\
						
						\hline
						$N_1^{\G_1}(a)$&$\left[a=-2,(g_{1,2} , g_{1,3})\not=(0,0)\right]$ or $[a=-\frac43]$  \\
						\hline																	$N_2^{\G_2^\al}(a)$&$[a=-1]$ or $[a=-\frac23]$  \\
						\hline

						$N_5^{\G_2^{-\frac29}}(a)$&$[a=-\frac13]$ or $[a=-\frac29]$  \\
						\hline
						$N_6^{\G_2^{-\frac29}}$, $N_7^{\G_2^{-\frac29}}$, $N_9^{\G_2^{0}}$, $N_{11}^{\G_2^{0}}$&{Always}\\
						$N_1^{\G_3}(a)$, $N_2^{\G_3}(a)$, $N_5^{\G_3}$&\\
						\hline
						$N_8^{\G_2^{0}}(a)$&$[a=-1]$ or $[a=-2]$\\
						\hline
						$N_{10}^{\G_2^{0}}(a)$&$[a=-1]$ \\
						\hline
						$N_1^{\G_4}(a)$, $N_1^{\G_5}(a)$&$[a=0]$ \\
						\hline
						
					\end{tabular}	\captionof{table}{Examples of Gauduchon Lie algebras.\label{table4}}\end{center}}																											{\renewcommand*{\arraystretch}{1.4}
				\begin{center}
					
					\begin{tabular}{|l|l|}
						\hline
						The Lie algebra& Conditions on the Hermitian metric	\\
						
						\hline

						\multirow{1}{*}	{	$N_2^{\G_3}(a)$ }&$\left[ g_{2,3}=0, g_{{2,2}}=g_{{3,3}},0<g_{{3,3}},a<-{\frac {g_{{1,2}} \left( g
								_{{1,2}}-g_{{3,3}} \right) }{{g_{{3,3}}}^{2}}},
						 g_{{1,1}}=-{\frac {ag_{{2,2}}g_{{3,3}}-g_{{1
										,2}}g_{{3,3}}-{g_{{1,3}}}^{2}
								}{g_{{3,3}}}}
						\right] 
						$ \\
						\hline
						\multirow{1}{*}	{	$N_5^{\G_3}$ }&Always \\
						\hline
						\multirow{2}{*}	{$N_1^{\G_4}(a)$,	$N_1^{\G_5}(a)$ }&\multirow{2}{*}	{$\left[ a=0,g_{{1,2}}=0,g_{{1,3}}=0, 0<g_{{1,1}},0<g_{{3,3}},{\frac {{g_{{2,3}}}^{2}}{
									g_{{3,3}}}}<g_{{2,2}} \right] $}\\
						&\\
						\hline
					\end{tabular}	\captionof{table}{Examples of infinitely  balanced Lie algebras.\label{table5} }\end{center}}

		{\renewcommand*{\arraystretch}{1.3}
			\begin{center}
				
				\begin{tabular}{|l|l|}
					\hline
					The Lie algebra& Conditions on the Hermitian metric	\\
					\hline																	
					\multirow{2}{*}	{	$N_2^{\G_2^\al}(a)$}&$\left[a=-1,\al=0,g_{12}=\frac{g_{1,3}g_{2,3}}{g_{3,3}}, 0<g_{{3,3}},{\frac {{g_{{1,3}}}^{2}}{
							g_{{3,3}}}}<g_{{1,1}},{\frac {{g_{{2,3}}}^{2}}{g_{{3,3}}}}<g_{{2,2}}
					   \right]$ \\
					&$\left[a=-1,g_{1,2}=g_{1,3}=0,(g_{2,2},g_{2,3})\not=(-\frac{g_{3,3}}{\al},\frac{g_{2,2}}2), 0<g_{{1,1}},0<g_{{3,3}},{\frac {{g_{{2,3}}}^{2}}{
							g_{{3,3}}}}<g_{{2,2}}  \right]$,\\
					\hline									
					\multirow{2}{*}	{	$N_5^{\G_2^{-\frac29}}(a)$}
						&$\left[ a=-\frac13,g_{{1,1}}=2\,g_{{1,3}},g_{{1,2}}=2\,g_{{1,3}},g_{{2,3}}=3\,g_{{1,3}},g_{{3,3}}=g_{{1
								,3}},0<g_{{1,3}},10\,g_{{1,3}}<g_{{2,2}} \right] 
						$\\
						&$\left[ a=-\frac13,g_{{1,1}}=2\,g_{{1,3}},g_{{1,2}}=3\,g_{{1,3}},g_{{2,3}
						}=3\,g_{{1,3}},g_{{3,3}}=g_{{1,3}}, 0<g_{{1,3}},9\,g_{{1,3}}<g_{{2,2}} \right] $\\
						\hline
						
						\multirow{1}{*}	{	$N_8^{\G_2^{0}}(a)$}&$\left[ a=-2,g_{{1,2}}=0,g_{{1,3}}=0,g_{{2,3}}\not=g_{{3,3}},0<g_{{1,1}},0<g_{{3,3}},{\frac {{g_{{2,3}}}^{2}}{
								g_{{3,3}}}}<g_{{2,2}} \right]$ \\
						\hline
						\multirow{1}{*}	{	$N_{10}^{\G_2^{0}}(a)$}&
							$\left[ a=-1,g_{{1,2}}=g_{{1,3}},g_{{2,3}}=g_{{3,3}}, 0<g_{{3,3}},g_{{3,3}}<g_{{2,2}},{\frac {{g_{{1,3}
										}}^{2}}{g_{{3,3}}}}<g_{{1,1}}
								 \right]$\\
							\hline
							\multirow{1}{*}	{	$N_{11}^{\G_2^{0}}$}&
							$\left[ g_{{1,2}}={\frac {g_{{1,3}}g_{{2,3}}}{g_{{3,3}}}},g_{{2,2}}={
								\frac {g_{{1,1}}g_{{2,3}}g_{{3,3}}-g_{{1,1}}{g_{{3,3}}}^{2}-{g_{{1,3}}
									}^{2}g_{{2,3}}+{g_{{1,3}}}^{2}g_{{3,3}}+{g_{{2,3}}}^{2}g_{{3,3}}}{{g_{
										{3,3}}}^{2}}},0<g_{{3,3}},g_{{3,3}}<g_{{2,3}},{\frac {{g_{{1,3}
									}}^{2}}{g_{{3,3}}}}<g_{{1,1}}
							 \right]$ \\
						\hline
						\multirow{1}{*}	{	$N_2^{\G_3}(a)$}&$\left[ g_{1,3}=g_{2,3}=0, g_{{1,1}}=-{\frac {ag_{{2,2}}g_{{3,3}}-g_{{1
										,2}}g_{{3,3}}
								}{g_{{3,3}}}},a=a,0<g_{{2,2}},0<g_{{3,3}},
							{\frac {{a}^
								{2}{g_{{2,2}}}^{2}-2\,ag_{{1,2}}g_{{2,2}}+{g_{{1,2}}}^{2}}{g_{{2,2}}}}
						<g_{{1,1}}
							\right] 
						$ \\
						\hline
						\multirow{1}{*}	{	$N_5^{\G_3}$}&Always \\
						\hline
						\multirow{1}{*}	{	$N_1^{\G_4}(a)$, $N_1^{\G_5}(a)$}
						&$\left[ a=0,g_{{1,2}}=0,g_{{1,3}}=0, 0<g_{{1,1}},0<g_{{3,3}},{\frac {{g_{{2,3}}}^{2}}{
								g_{{3,3}}}}<g_{{2,2}} \right]$\\
						\hline

					\end{tabular}	\captionof{table}{Examples of balanced non K\"ahler Lie algebras.\label{table6}}\end{center}}																
    
          	{\renewcommand*{\arraystretch}{1.4}
          		\begin{center}
          			
          			\begin{tabular}{|l|l|}
          				\hline
          				The Lie algebra& Conditions on the Hermitian metric	\\
          				
          				\hline
          				$N_1^{\G_1}(a)$& $\left[a=-2,(g_{1,2} , g_{1,3})\not=(0,0)\right]$ or $[a=-1]$.\\
          				\hline																			
          				$N_2^{\G_1}$& Always \\
          				\hline
          				\multirow{4}{*}	{$N_2^{\G_2^\al}(a)$}& 
          				$\left[a=-\frac23,\al=-\frac29,g_{1,3}=0,g_{3,3}=-\frac{2g_{2,2}}{9}+g_{2,3},0<g_{{2,3}},g_{{2,2}}<3\,g_{{2,3}},\right.$\\
          				&$\left.{
          					\frac {{g_{{1,2}}}^{2} \left( 2\,g_{{2,2}}-9\,g_{{2,3}} \right) }{2\,{
          							g_{{2,2}}}^{2}-9\,g_{{2,2}}g_{{2,3}}+9\,{g_{{2,3}}}^{2}}}<g_{{1,1}},
          				\frac32
          				\,g_{{2,3}}<g_{{2,2}}
          				  \right]$\\

          				&$\left[ a\in\{-\frac12,-1\},g_{1,3}=0,g_{2,3}=\frac12g_{2,2},g_{3,3}=-\al g_{2,2},0<g_{{2,2}},\alpha<-\frac14\,{\frac {g_{{2,2}}g_{{1,1}}}{g_{{2,2}
          						}g_{{1,1}}-{g_{{1,2}}}^{2}}},{\frac {{g_{{1,2}}}^{2}}{g_{{
          							2,2}}}}<g_{{1,1}}
          			 \right]$\\
          				\hline																			\multirow{1}{*}	{$N_2^{\G_2^{a^2+a}}(a)$} 
          				
          				&$\left[a=-\frac23,g_{3,3}=-\frac29g_{2,2}+g_{2,3},g_{1,2}=0,0<g_{{2,3}},g_{{2,2}}<3\,g_{{2,3}},\frac32\,g
          				_{{2,3}}<g_{{2,2}},-\,{\frac {9{g_{{1,3}}}^{2}g_{{2,2}}}{2\,{g_{{2,2}}
          						}^{2}-9\,g_{{2,2}}g_{{2,3}}+9\,{g_{{2,3}}}^{2}}}<g_{{1,1}} 
          				 \right]$\\
          				
          				\hline
          				\multirow{2}{*}	{$N_6^{\G_2^{-\frac29}}$}&\multirow{2}{*}{
          					$\left[g_{{2,2}}=\frac92(\,g_{{2,3}}-\,g_{{3,3}}), g_{1,2}=0, 0<g_{{2,3}},g_{{2,2}}<3\,g_{{2,3}},\frac32\,g
          					_{{2,3}}<g_{{2,2}},-\,{\frac {9{g_{{1,3}}}^{2}g_{{2,2}}}{2\,{g_{{2,2}}
          							}^{2}-9\,g_{{2,2}}g_{{2,3}}+9\,{g_{{2,3}}}^{2}}}<g_{{1,1}} 
          					 \right]$}\\
          				&	\\
          				\hline
          				
          				\multirow{1}{*}	{$N_8^{\G_2^{0}}(a)$, $N_9^{\G_2^{0}}$}&
          				$\left[ a\in\{-1,-2\},g_{1,3}=0,g_{{2,3}}=g_{{3,3}}, 0<g_{{1,2}},0<g_{{3,3}},g_{{1,2}}<g_{{3,3}},g_{{3,3}}<g_{{2,2
          					}},{\frac {{g_{{1,2}}}^{2}}{g_{{2,2}}-g_{{3,3}}}}<g_{{1,1}}
          					 \right]$
          				\\
          				\hline
          				\multirow{2}{*}	{$N_{10}^{\G_2^{0}}(a)$}&
          				$\left[ a=-2,g_{{1,2}}=g_{{1,3}}+2\,g_{{2,2}}-2\,g_{{2,3}},g_{{3,3}}=
          				g_{{2,3}}, 0<g_{{3,3}},g_{{3,3}}<g_{{2,2}},{\frac {{
          							g_{{1,3}}}^{2}+4\,g_{{2,2}}g_{{3,3}}-4\,{g_{{3,3}}}^{2}}{g_{{3,3}}}}<g
          				_{{1,1}}
          				 \right] 
          				$
          				\\
          				&$\left[ a=-1,g_{{1,2}}=g_{{1,3}}+g_{{2,2}}-g_{{2,3}},g_{{3,3}}=g_{{2,
          						3}}, 0<g_{{3,3}},g_{{3,3}}<g_{{2,2}},{\frac {{
          							g_{{1,3}}}^{2}+g_{{2,2}}g_{{3,3}}-{g_{{3,3}}}^{2}}{g_{{3,3}}}}<g_{{1,1
          					}}
          					 \right] $\\
          				\hline
          				\multirow{1}{*}	{$N_{11}^{\G_2^{0}}$}&
          				$\left[ g_{{1,2}}=g_{{1,3}}+g_{{2,2}}-g_{{2,3}},g_{{3,3}}=g_{{2,
          						3}}, 0<g_{{3,3}},g_{{3,3}}<g_{{2,2}},{\frac {{
          							g_{{1,3}}}^{2}+g_{{2,2}}g_{{3,3}}-{g_{{3,3}}}^{2}}{g_{{3,3}}}}<g_{{1,1
          					}}
          					\right] $          				\\
          				\hline
          				\multirow{1}{*}	{$N_2^{\G_3}(a)$}&$[a=1]$
          				\\
          				\hline
          				
          				\multirow{1}{*}	{$N_1^{\G_5}(a)$}&
          				$\left[ a=0,g_{{2,2}}=g_{{3,3}},g_{{2,3}}=0,0<g_{{2,2}},{\frac {{
          							g_{{1,2}}}^{2}+{g_{{1,3}}}^{2}}{g_{{2,2}}}}<g_{{1,1}} 
          				 \right] $
          				\\
          				\hline	
          			\end{tabular}	\captionof{table}{Examples of pluriclosed non K\"ahler Lie algebras.\label{table7}}\end{center}}

																													{\renewcommand*{\arraystretch}{1.4}
																														\begin{center}
																															
																															\begin{tabular}{|l|l|}
												\hline
												The Lie algebra& Conditions on the Hermitian metric	\\
																																\hline
																																\multirow{1}{*}{$N_1^{\G_1}(a)$}&$\left[a=2,g_{{1,2}}=0,g_{{1,3}}=0, 0<g_{{1,1}},0<g_{{3,3}},{\frac {{g_{{2,3}}}^{2}}{g_{{3,3}}}}<g_{{2,2}}
\right]$\\		\hline
																																	\multirow{4}{*}{$N_2^{\G_2^\al}(a)$}&$\left[ a=1,\alpha=2,g_{{1,2}}=-\,{\frac {2g_{{2,2}}g_{{1,3}}}{g_{{3,	3}}}},g_{{2,3}}=0,0<g_{{2,2}},0<g_{{3,3}},{\frac { \left( 4	\,g_{{2,2}}+g_{{3,3}} \right) {g_{{1,3}}}^{2}}{{g_{{3,3}}}^{2}}}<g_{{1,1}}
\right] 	$  \\
																																										& $\left[ a=1,g_{{1,2}}=0,g_{{1,3}}=0, 0<g_{{1,1}},0<g_{{3,3}},{\frac {{g_{{2,3}}}^{2}}{g_{{3,3}}}}<g_{{2,2}} \right]$\\ 
																																	&$\left[ a=1,\alpha=2,g_{{1,2}}=0,g_{{2,3}}=2\,g_{{2,2}},0<g_{{2,2}},g_{{1,3}}<-\frac52\,g_{{2,2}},4\,g_{{2,2}}<g_{{3,3}},
		-{\frac {{g_{{1,3}}}^{2}}{-g_{{3,3}}+4\,g_{{2,2}}}}<g_{{1,1}}
 \right]$\\																														&$\left[ a=1,\alpha=2,g_{{1,3}}=-g_{{1,2}},g_{{2,2}}=-\frac12(\,g_{{2,3}}-
																																\,g_{{3,3}}), 0<g_{{3,3}},g_{{2,3}}<\frac12\,g_{{3,3}},-g_{{3,3}}<g_{{2,3}},-\,{\frac {3{g_{{1,2}}}^{2}}{2\,g_{{2,3}}-g_{{3,3}}}}	<g_{{1,1}}
																																 \right]$\\ 
																																	
\hline
																																																																		\multirow{2}{*}{$N_5^{\G_2^{-\frac29}}(a)$}&
																																	$\left[ a=\frac13,g_{1,2}=0,g_{{1,1}}={\frac {3\,{g_{{1,2}}}^{2}+g_{{1,2}}g_{{2,3}}	-2\,{g_{{2,3}}}^{2}}{3\,g_{{2,2}}-9\,g_{{2,3}}}},g_{{1,3}}=-\frac13\,g_{{2
																							,3}},g_{{3,3}}=\frac13\,g_{{2,3}}, 0<g_{{2,3}},3\,g_{{2,3}}<g_{{2,2}},\right.$\\&$\left.\frac43\,{\frac {g_{{2,2}}{g_{
																									{2,3}}}^{3}}{{g_{{2,2}}}^{3}-9\,{g_{{2,2}}}^{2}g_{{2,3}}+27\,g_{{2,2}}
																							{g_{{2,3}}}^{2}-27\,{g_{{2,3}}}^{3}}}<g_{{1,1}}
 \right]$\\

																																		\hline 
																																																																																																	\multirow{1}{*}{$N_8^{\G_2^{0}}(a)$}&  $\left[ a=1,g_{{1,2}}=0,g_{{1,3}}=0, 0<g_{{3,3}},{\frac {{g_{{2,3}}}^{2}}{g_{{3,3}}}}<g_{{2,2}} \right]$\\
																																		\hline
																																		\multirow{3}{*}{$N_{10}^{\G_2^{0}}(a)$}																											
&$\left[ a=1,g_{{1,1}}=-g_{{1,2}},g_{{1,3}}=0,g_{{2,3}}=g_{{3,3}}, 0<g_{{2,3}},g_{{1,2}}<0,-g_{{1,2}}+g_{{2,3}}<g_{{2,2}}

																																								\right]$\\
																																								&$\left[ a=1,g_{{1,1}}=-{\frac {g_{{1,2}}g_{{3,3}}-{g_{{1,3}}}^{2}-g_{
																																											{1,3}}g_{{2,3}}}{g_{{3,3}}}},g_{{2,2}}={\frac {g_{{1,2}}g_{{2,3}}-g_{{
																																												1,2}}g_{{3,3}}+g_{{1,3}}g_{{2,3}}}{g_{{1,3}}}},0<g_{{3,3}},g_{{1,2}}<{\frac {g_{{1,3}}g_{{2,3}}}{g_{{3,3}}}},\right.$\\
																																								&$\left.g_{{1,3}}<0,g_{{3,3}}<-2\,g_{{1,3}}-\sqrt {2}\sqrt {{g_{{1,3}}}^{2}},
																																								g_{{3,3}}-g_{{1,3}}<g_{{2,3}}
		 \right]$\\
																																								\hline

																																							\end{tabular}	\captionof{table}{Examples of Calabi-Yau with trosion Lie algebras which are not infinitely balanced\label{table8} }\end{center}}

\end{document}